\documentclass{article} 
\usepackage{graphicx}
\usepackage{amsmath}
\usepackage{amsthm}
\usepackage[T1]{fontenc}
\usepackage{amssymb}
\usepackage[utf8]{inputenc}
\usepackage{graphicx}
\usepackage{indentfirst}
\usepackage[hyphens]{url} 
\usepackage[english]{babel}	
\usepackage{hyperref}			
\usepackage{color}
\usepackage{float} 
\usepackage{tikz}
\usetikzlibrary{arrows}
\usepackage{caption}
\usepackage{subcaption}
\usepackage{babel} 
\usepackage[justification=centering]{caption} 

\addto\captionsfrench{}

\def\Hk1{\mathrm{H}^{k+1}}
\def\HH1{\mathrm{H}^1}
\def\Hexpo{\mathrm{H}^}
\def\L2{\mathrm{L}^2}  
\def\LL{\mathrm{L}^}  
\def\Linf{\mathrm{L}^\infty} 
\def\H2{\mathrm{H}^2}

\def\c{\mathcal{C}^} 
\def\P{\mathbb{P}^}
\def\omgam{(\Omega,\Gamma)}

\def\Vh{\mathbb{V}_h}

\def\Vhlifte{\mathbb{V}_h^\ell}

\def\d{\mathrm{d}}
\def\hessienne{\mathcal{H}}
\def\nn{\boldsymbol{\mathrm{n}}}

\def\nnhr{\boldsymbol{\mathrm{n}_{hr}}}

\def\nt{\nabla_\Gamma}
\def\na{\nabla}
\def\div{\mathrm{div}}
\def\dist{\mathrm{dist}}
\def\diff{\mathrm{D}}

\def\r{^{(r)}}

\def\omhh{\Omega_h}
\def\omh1{\Omega_h^{(1)}}
\def\omhr{\Omega_h^{(r)}}

\def\ghh{\Gamma_h}
\def\gh1{\Gamma_h^{(1)}}
\def\ghr{\Gamma_h^{(r)}}

\def\ft{F_T}
\def\fte{F_T^{(e)}}
\def\ftr{F_T^{(r)}}
\def\ftre{F_{T^{(r)}}^{(e)}}

\def\Ghr{G_h^{(r)}}
\def\Ahell{A_h^\ell}
\def\ahell{a_h^\ell}

\def\Jb{J_b}
\def\Jh{J_h}
\def\Jblifte{J_b^\ell}
\def\Jhlifte{J_h^\ell}

\def\tref{\hat{T}}
\def\trefminissigma{\tref\backslash\hat{\sigma}}
\def\hatsigma{\hat{\sigma}}

\def\te{{T}^{(e)}}

\def\tdeux{{T}^{(2)}}
\def\tr{{T}^{(r)}}

\def\tauh{\mathcal{T}_h^{(1)}}

\def\taur{\mathcal{T}_h^{(r)}}
\def\taue{\mathcal{T}_h^{(e)}}

\def\lambdaetoile{\lambda^*}
\def\hatx{\hat{x}}
\def\haty{\hat{y}}
\def\hatv{\hat{v}}
\def\rhotr{\rho_{\tr}}

\newcommand{\fonction}[5]{\begin{array}[t]{lrcl}#1 :&#2 &\longrightarrow &#3\\&#4& \longmapsto &#5 \end{array}}

\newtheorem{theorem}{Theorem}
\numberwithin{theorem}{section}

\newtheorem{lem}[theorem]{Lemma}

\newtheorem{corollary}[theorem]{Corollary}
\newtheorem{proposition}[theorem]{Proposition}
\newtheorem{definition}[theorem]{Definition}

\newtheorem{remark}[theorem]{Remark}

\addto\captionsfrench{}

\def\G{\mathcal{G}_h^{(r)}}

\def\Ghdeux{{G}_h^{(2)}}

\def\R{\mathbb{R}}      
\def\N{\mathbb{N}}  
\def\I{\mathrm{Id}}
\def\transpose{^\mathsf{T}}

\def\Ih1{\mathcal{I}^{(1)}}
\def\Ihr{\mathcal{I}^{(r)}}
\def\Ihlifte{\mathcal{I}^\ell}

\begin{document}

\title{A priori error estimates of a Poisson equation with Ventcel boundary conditions on curved meshes}
\author{Fabien Caubet\footnote{Universit\'e de Pau et des Pays de l'Adour, E2S UPPA, CNRS, LMAP, UMR 5142, 64000 Pau, France. \texttt{fabien.caubet@univ-pau.fr}}, 
Joyce Ghantous\footnote{Universit\'e de Pau et des Pays de l'Adour, E2S UPPA, CNRS, LMAP, UMR 5142, 64000 Pau, France. \texttt{joyce.ghantous@univ-pau.fr}},
 Charles Pierre\footnote{Universit\'e de Pau et des Pays de l'Adour, E2S UPPA, CNRS, LMAP, UMR 5142, 64000 Pau, France. \texttt{charles.pierre@univ-pau.fr}}
}
\date{ }

\maketitle

\begin{abstract}
In this work is considered {an elliptic problem}, referred to as the \textit{Ventcel problem}, involving a second order term on the domain boundary (the Laplace-Beltrami operator). A variational formulation of the Ventcel problem is studied, leading to a finite element discretization. The focus is on the construction of high order curved meshes for the discretization of the physical domain and on the definition of the lift operator, which is aimed to transform a function defined on the mesh domain into a function defined on the physical one. This \textit{lift} is defined in a way as to satisfy adapted properties on the boundary, relatively to the trace operator. The Ventcel problem approximation is investigated both  in terms of geometrical error and of finite element approximation error. Error estimates are obtained both in terms of the  mesh order $r\ge 1$ and to the finite element degree $k\ge 1$, whereas such estimates usually have been considered in the isoparametric case so far, involving a single parameter $k=r$. The numerical experiments we led, both in dimension 2 and 3, allow us to validate the results obtained and proved on the \textit{a priori} error estimates depending on the two parameters $k$ and~$r$. A numerical comparison is made between the errors using the former lift definition and the lift defined in this work establishing an improvement in the convergence rate of the error in the latter case.
\end{abstract}

\textbf{{keywords}}: Laplace-Beltrami operator, Ventcel boundary condition, finite element method, high order meshes, geometric error, \textit{a priori} error estimates. \medskip

\textbf{MSCcodes}: 
74S05, 65N15, 65N30, 65G99.

\section{Introduction}
\paragraph{Motivations.} 
In various situations, we have to numerically solve a Partial Differential Equation (PDE), typically with a finite element method, on smooth geometry. A key point is to obtain an estimation of the error produced while approximating the solution~$u$ of the problem, by its finite element approximation~$u_h$ while taking into account the error produced while approximating the physical domain $\Omega$ by the mesh domain $\Omega_h$. \medskip

This typically is the case in this work, which is aimed at certain industrial applications (in particular in the context of the project RODAM\footnote{\textit{Robust Optimal Design under Additive Manufacturing constraints}: \url{https://lma-umr5142.univ-pau.fr/en/scientific-activities/scientific-challenges/rodam.html}.}) where the  object or material under consideration is surrounded by a thin layer with different properties, typically a corrosion layer. Another application is also observed in aeroacoustic, where the so-called Ingard-Myers boundary conditions are used to model the presence of a liner located on the surface of a duct (see \cite{aeroacoustics}). The presence of this layer causes some difficulties while discretizing the domain and numerically solving the problem. To overcome this problem, a classical approach consists in replacing the thin layer by a model with artificial boundary conditions. When considering diffusivity properties, this leads to introduce second-order boundary conditions, the so-called {\it Ventcel boundary conditions}, as analysed in \cite{Gvial}. In the second half of the 1950's, these conditions were introduced in the pioneering works of Ventcel \cite{Ventcel-1956,Ventcel-1959}. The price to pay is to impose the smoothness of the domain in order to guaranty the well posedness of the second order boundary condition, which implies that the physical domain cannot be fitted by a polygonal mesh. \medskip

To sum up, the main focus of this paper is to consider the numerical resolution of a (scalar) PDE equipped with higher order boundary conditions, which are the Ventcel boundary conditions, to after that assess the {\it a priori} error produced by a finite element approximation, on higher order meshes.

\medskip

\paragraph{The Ventcel problem and its approximation.} Let $\Omega$ be a nonempty bounded connected domain in $\R^{d}$, $d=2$, $3$, with a smooth boundary $\Gamma := \partial\Omega$. Considering {the source terms $f$ and $g$}, as well as some given constants $\kappa \ge 0$, $\alpha,\,\beta>0$, the Ventcel problem that we will focus on is the following:
\begin{equation} \label{1}
  \left\{
      \begin{array}{rcll}
       -\Delta u + \kappa u &=& f &   \text{ in } \Omega,\\
        -\beta \Delta_{\Gamma} u + \partial_{\mathrm{n}} u + \alpha u &=& g &  \text{ on } \Gamma,\\
      \end{array}
    \right.
\end{equation}
where $\nn$ denotes the external unit normal to $\Gamma$, $\partial_{\mathrm{n}} u$ the normal derivative of $u$ along~$\Gamma$ and~$\Delta_\Gamma$ the Laplace-Beltrami operator. 
\medskip

The main objective of this work is  to do an error analysis of the Ventcel Problem. To begin with, we need to point out that the domain $\Omega$ is required to be smooth due to the presence of second order boundary conditions. Actually, Ventcel boundary conditions would not make sense on polygonal domains. Thus, the physical domain~$\Omega$ being non-polygonal can not be exactly fitted by the mesh domain, i.e. $\Omega_h \neq \Omega$. This gap between $\Omega$ and the mesh domain produces a \textit{geometric error}. When using classical meshes made of triangles (affine meshes), this geometric error induces a saturation of the error at low order, independently of the considered finite element order. To overcome this issue, we will resort to curved meshes, following the work of many authors (see, e.g., \cite{PHcia,ciaravtransf,ed,elliott}). Meshes of order $r$ (\textit{i.e.} with elements of polynomial degree $r$) will be considered to improve the asymptotic behavior of the geometric error with respect to the mesh size $h$. Notice that the domain of the mesh of order $r$, denoted~$\omhr$, does not fit the domain~$\Omega$. However, the numerical results are expected to be more accurate for $r\ge 2$ than for standard affine meshes.

\medskip 

A $\P k$-Lagrangian finite element method is used with a degree~$k \ge 1$ to approximate the exact solution $u$ of System~\eqref{1} by a finite element function~$u_h$ defined on the mesh domain $\omhr$. One goal of the present paper is to perform an error analysis both considering the roles of the finite element approximation error, controlled by the parameter $k$,  and the geometric error, controlled by the parameter $r$. We thus consider a non-isoparametric approach, in the sequel of the work of Demlow \textit{et al.} for surface problems as precised later on. Doing so, one can assess which is the optimal degree of the finite element method $k$ to chose depending on the geometrical degree~$r$, in order to minimize the total error. Notice that an \textit{isoparametric approach}, that is taking $k=r$, is treated in~\cite{ed,elliott,Balazs}, for similar problems. 

\medskip

Since $\omhr \neq \Omega$, in order to compare the numerical solution $u_h$ defined on $\omhr$ to the exact solution $u$ defined on $\Omega$ and to obtain \textit{a priori} error estimations, the notion of \textit{lifting} a function from a domain onto another domain needs to be introduced. The \textit{lift functional} was firstly introduced in the 1970s by many authors (see, e.g.,~\cite{dubois,Lenoir1986,nedelec,scott-2}). Among them, let us emphasize the lift based on the orthogonal projection onto the boundary $\Gamma$, introduced by Dubois in~\cite{dubois} and further improved in terms of regularity by Elliott \textit{et al.} in~\cite{elliott}. However, the lift defined in \cite{elliott} does not fit the orthogonal projection on the computational domain's boundary. As will be seen in Section~\ref{sec:lift-def-surf-vol}, this condition is essential to guarantee  the theoretical analysis of this problem. In order to address this issue, an alternative definition is introduced in this paper which will be used to perform a numerical study of the computational error of System~\eqref{1}. This modification in the lift definition has a big impact on the error approximation as is observed in the numerical examples in Section~\ref{sec:numerical-ex}. 

\paragraph{Main novelties.} 
The first innovating point presented in this work, is the definition of a new adequate lift satisfying a suitable \textit{trace property}, as developed in Proposition~\ref{prop:trace}. The second novelty in this paper is the {\it a priori} error estimations, which are computed and expressed both in terms of finite element approximation error and of geometrical error, respectively, associated to the finite element degree $k \ge 1$ and to the mesh order $r \ge 1$. This follows the works of Demlow \cite{D4,D1,D2} on surface problems, where he considered a non isoparametric approach with $k \ne r$, in order to do an error analysis. In the existing works such as~\cite{ed}, error estimates of Problem~\eqref{1} were established using the lift defined in \cite{elliott}, while considering an \textit{isoparametric approach} and taking~$k=r$. In~\cite{elliott}, while also taking an \textit{isoparametric approach}, a thorough error analysis is made on a coupled bulk–surface partial differential equation with Ventcel boundary conditions. In~\cite{ventcel1}, the well-posedness and regularity of System~\eqref{1} is rigorously studied. Eventually, this paper also brings to the fore an interesting super convergence property of quadratic meshes, numerically observed both in dimension 2 and 3.

We present the following \textit{a priori} error estimations, which will be explained in details and proved in Section~\ref{ERROR-section}:
$$
  \| u-u_h^\ell \|_{\L2 \omgam } = O ( h^{k+1} + h^{r+1})
\quad  {\rm and } \quad  
  \| u- u_h^\ell \|_{\HH1 \omgam } = O ( h^{k} + h^{r+1/2}),
$$
where $h$ is the mesh size and $u_h^\ell$ denotes the \textit{lift} of $u_h$ (given in Definition~\ref{def:liftvolume}), and~$\L2 \omgam$ and~$\HH1 \omgam$ are Hilbert spaces defined below.

\medskip

\paragraph{Paper organization.} Section \ref{sec:notations_def} contains all the mathematical tools and useful definitions to derive the weak formulation of System~\eqref{1}. Section~\ref{sec:mesh} is devoted to the definition of the high order meshes. In Section \ref{section-lift}, are defined the volume and surface lifts, which are the keystones of this work. A Lagrangian finite element space and discrete formulation of System~\eqref{1} are presented in Section~\ref{FE-section}, alongside their \textit{lifted forms} onto $\Omega$. The \textit{a priori} error analysis is detailed in Section~\ref{ERROR-section}. The paper wraps up in Section~\ref{sec:numerical-ex} with {2D and 3D} numerical experiments studying the method convergence rate dependency on the geometrical order $r$ and on the finite element degree $k$. 
\section{Notations and needed mathematical tools}
\label{sec:notations_def}
Firstly, let us introduce the notations that we adopt in this paper. Throughout this paper,~$\Omega$ is a nonempty bounded connected open subset of $\R^{d}$ $(d=2,3)$ with a smooth (at least $\c2$) boundary~$\Gamma:=\partial{\Omega}$. 
The unit normal to~$\Gamma$ pointing outwards is denoted by~$\nn$ and $\partial_{\mathrm{n}} u$ is a normal derivative of a function $u$.
We denote respectively by $\LL 2(\Omega)$ 
and  $\LL 2(\Gamma)$ 
the usual Lebesgue spaces endowed with their standard norms on $\Omega$ and $\Gamma$.
Moreover, for $k \geq 1$, $\Hk1(\Omega)$ denotes the usual Sobolev space endowed with its standard norm. We also consider the Sobolev spaces~$\Hk1(\Gamma)$ on the boundary as defined e.g. in \cite[\S 2.3]{ventcel1}. It is recalled that the norm on $\Hexpo{1}(\Gamma)$ is: $\|u\|^2_{\Hexpo{1}(\Gamma)} : =  \|u\|^2_{\L2(\Gamma)} + \|\nt u\|^2_{\L2(\Gamma)},$ where $\nt$ is the tangential gradient defined below; and that $\|u\|^2_{\Hk1(\Gamma)} :=  \|u\|^2_{\Hexpo{k}(\Gamma)} + \|\nt u\|^2_{\Hexpo{k}(\Gamma)}$. 
Throughout this work, we rely on the following Hilbert space (see \cite{ventcel1}) 
$$
    \Hexpo{1}\omgam := \{ u \in \Hexpo{1}(\Omega), \ u_{|_\Gamma} \in \Hexpo{1}(\Gamma) \},
$$
equipped with the norm $\|u\|^2_{\Hexpo{1}\omgam} : =  \|u\|^2_{\Hexpo{1}(\Omega)} + \| u\|^2_{\Hexpo{1}(\Gamma)}.$
In a similar way is defined the following space $\L2 \omgam := \{ u \in \L2(\Omega), \ u_{|_\Gamma} \in \L2(\Gamma) \},$ equipped with the norm~$\|u\|^2_{\L2\omgam} :=  \|u\|^2_{\L2(\Omega)} + \| u\|^2_{\L2(\Gamma)}$. More generally, we define $\Hk1 \omgam:=\{ u \in \Hk1(\Omega), \ u_{|_\Gamma} \in \Hk1(\Gamma) \}$. 

Secondly, we recall the definition of the tangential operators (see, e.g.,~\cite{livreopt}). 
\begin{definition}
Let $w  \in \HH1(\Gamma)$, $W \in \HH1(\Gamma,\R^d)$ and $u \in \H2 (\Gamma)$. Then the following operators are defined on $\Gamma$: 
\begin{itemize}
    \item the tangential gradient of $w$ given by $ \nt w :=\na \tilde{w}  - (\na \tilde{w} \cdot \nn)\nn $, where $\tilde{w} \in \HH1(\R^d)$ is any extension of $w$;
\item the tangential divergence of $W$ given by $ \div_{\Gamma} W : = \div \tilde{W}- (\diff \tilde{W} \nn) \cdot \nn$, where~$\tilde{W} \in \HH1(\R^d, \R^d)$ is any extension of $W$ and $\diff\tilde{W} = (\na \tilde{W}_i)_{i=1}^d$ is the differential matrix of the extension $\tilde{W}$;
\item the Laplace-Beltrami operator of $u$ given by 
$ \Delta_\Gamma u  := \div_\Gamma (\nt u)$.
\end{itemize}
\end{definition}

\medskip

Additionally, the constructions of the mesh used in Section \ref{sec:mesh}  and of the lift procedure presented in Section~\ref{section-lift} are based on the following fundamental result that may be found in \cite{tubneig} and \cite[\S 14.6]{GT98}. For more details on the geometrical properties of the tubular neighborhood and the orthogonal projection defined below, we refer to~\cite{D1,D2,actanum}.
\begin{proposition}
\label{tub_neigh_orth_proj_prop}
Let $\Omega$ be a nonempty bounded connected open subset of $\R^{d}$ 
with a~$\c2$ boundary~$\Gamma= \partial \Omega$. Let $\d : \R^d \to \R$ be the signed distance function with respect to $\Gamma$ defined by,
\begin{displaymath}
  \d(x) :=
  \left \{
    \begin{array}{ll}
      -\dist(x, \Gamma)&  {\rm if } \, x \in \Omega ,
      \\
      0&  {\rm if } \, x \in \Gamma ,
      \\
      \dist(x, \Gamma)&  {\rm otherwise},
    \end{array}
  \right. \qquad 
  {\rm with} \quad \dist(x, \Gamma) := \inf \{|x-y|,~ \ y \in \Gamma \}.
\end{displaymath}
Then there exists a tubular neighborhood $\mathcal{U}_{\Gamma}:= \{ x \in \R^d ; |\d(x)| < \delta_\Gamma \}$ of $\Gamma$, of sufficiently small width $\delta_\Gamma$, where {$\d$ is a $\c2$ function}. Its gradient $\na \d$ is an extension of the external unit normal~$\nn$ to $\Gamma$. Additionally, in this neighborhood~$\mathcal{U}_{\Gamma}$, the orthogonal projection~$b$ onto $\Gamma$ is uniquely defined and given by,
\begin{displaymath} 
b\, :~ x \in \mathcal{U}_{\Gamma}  
\longmapsto    b(x):=x-\d(x)\na \d(x) \in \Gamma.
\end{displaymath}
\end{proposition}

%
%
%
%
%
%
%

\medskip

Finally, the variational formulation of Problem~\eqref{1} is obtained, using the integration by parts formula on the surface~$\Gamma$ (see, e.g. \cite{livreopt}), and is given by,
\begin{equation}
\label{fv_faible}
    \mbox{find } u \in \HH1 \omgam \mbox{ such that }  a(u,v) = l(v), \,  \forall \ v \in \HH1\omgam,
\end{equation}
where the bilinear form $a$, defined on $\HH1\omgam^2$, is given by,
$$
    a(u,v) := \int_{\Omega} \nabla u \cdot \nabla v \, \d x +\kappa \int_{\Omega}  u  v \, \d x + \beta \int_{\Gamma} \nabla_{\Gamma} u \cdot \nabla_{\Gamma} v \, \d\sigma + \alpha \int_{\Gamma} u  v \, \d\sigma,
$$
and the linear form $l$, defined on $\HH1\omgam$, is given by,
$$
    l(v) :=  \int_{\Omega} f v \, \d x +\int_{\Gamma} g v \, \d\sigma.
$$
The following theorem claims the well-posedness of the problem \eqref{fv_faible} proven in \cite[th. 2]{Jaca} and \cite[th. 3.3]{ventcel1} and establishes the solution regularity proven in \cite[th. 3.4]{ventcel1}.
\begin{theorem}
\label{th_existance_unicite_u}
Let $\Omega$ and  $\Gamma= \partial \Omega$ be as stated previously. Let $\alpha$, $\beta >0$, $\kappa \ge 0 $, and $f \in \L2(\Omega)$,~$g \in \L2(\Gamma)$. Then there exists a unique solution $u \in \Hexpo{1}\omgam$ to problem (\ref{fv_faible}).

Moreover, if $\Gamma$ is of class $\c {k+1}$, and $f \in \Hexpo{k-1}(\Omega)$, $g \in \Hexpo{k-1}(\Gamma)$, then the solution~$u$ of \eqref{fv_faible} is in~$\Hk1 \omgam$ and is the strong solution of the Ventcel problem~\eqref{1}. Additionally, there exists~$c>0$ such that the following inequality holds, 
$$
    \|u\|_{\Hk1\omgam} \le c ( \|f\|_{\Hexpo{k-1}(\Omega)} + \|g\|_{\Hexpo{k-1}(\Gamma)}).
$$
\end{theorem}

%
%
\section{Curved mesh definition}
\label{sec:mesh}

In this section we briefly recall the construction of curved meshes of geometrical order~$r\ge 1$ of the domain~$\Omega$ and introduce some notations.  We refer to~\cite[Section~2]{Jaca} for details and examples (see also~\cite{elliott,scott-2,dubois,Bernardi1989}). Recall for $r\ge 1$, the set of polynomials in~$\R^d$ of order $r$ or less is denoted by $\P r$. 
From now on, the domain~$\Omega$, is assumed to be at least $\c {r+2}$ regular, and~$\tref$ denotes the reference simplex of dimension~$d$. 
In a nutshell, the way to proceed is the following. 
\begin{enumerate}
    \item Construct an affine mesh $\tauh$ of $\Omega$ composed of simplices  $T$ and define the affine transformation $\ft:~\tref\rightarrow T:=\ft(\tref)$ associated to each simplice $T$. 
    \item For each simplex $T \in \tauh$,  a mapping $\fte:~\tref\rightarrow \te := \fte(\tref)$ is designed and the resulting {\it exact elements}  $\te$ will form a curved exact mesh $\taue$ of~$\Omega$. 
    \item For each $T \in \tauh$, the mapping  $\ftr$ is the $\P r$ interpolant of $\fte$. The curved mesh $\taur$ of order $r$ is composed of the elements  $\tr := \ftr(\tref)$.
\end{enumerate}
%
%
%
%
%
%
%
%
\subsection{Affine mesh $\tauh$} 
Let $\tauh$ be a polyhedral mesh of $\Omega$ made of simplices of dimension $d$ (triangles or tetrahedra), it is chosen as quasi-uniform and henceforth shape-regular (see  \cite[definition 4.4.13]{quasi-unif}). 
Define the mesh size $h:= \max\{\mathrm{diam}(T); \; T \in \tauh \}$, where $\mathrm{diam}(T)$ is the diameter of $T$. The mesh domain is denoted by $\omh1:= \cup_{T\in  \tauh}T$. Its boundary denoted by $\gh1 :=\partial \omh1$ is composed of $(d-1)$-dimensional simplices that form a  mesh of $\Gamma = \partial \Omega$. The vertices of $\gh1$ are assumed to lie on~$\Gamma$. 

For $T \in \tauh$, we define an affine function that maps the reference element onto~$T$, 
$$
\ft : \tref \to T:=\ft(\tref).
$$
\begin{remark}
For a sufficiently small mesh size $h$, the mesh boundary satisfies $\gh1 \subset \mathcal{U}_\Gamma$, 
where~$\mathcal{U}_{\Gamma}$ is the tubular neighborhood 
given in proposition~\ref{tub_neigh_orth_proj_prop}.
This guaranties that the orthogonal projection~$b: \gh1\rightarrow \Gamma$ is one to one which is 
required for the construction of the exact mesh. 
\end{remark}

\subsection{Exact mesh $\taue$}
In the 1970's, Scott gave an explicit construction of an exact triangulation in two dimensions in~\cite{scott-2}, generalised by Lenoir in~\cite{Lenoir1986} afterwards 
(see also \cite[\S 4]{elliott} and \cite[\S 3.2]{ed}). 
The present definition of an exact transformation~$\fte$ combines the definitions found in \cite{Lenoir1986,scott-2} with the projection $b$ as used in \cite{dubois}. 

\medskip

Let us first point out that for a sufficiently small mesh size $h$, a mesh element $T$ cannot have~$d+1$ vertices on the boundary $\Gamma$, due to the quasi uniform assumption imposed on the mesh $\tauh$. A mesh element is said to be an internal element if it has at most one  vertex on the boundary $\Gamma$.

\begin{definition}
\label{def:sigma-lambdaetoile-haty}
    Let $T\in\tauh$ be a non-internal element (having at least 2 vertices on the boundary). Denote~$v_i = \ft(\hatv_i)$ as its vertices, where $\hatv_i$ are the vertices of~$\tref$. We define $\varepsilon_i=1$  if $v_i\in  \Gamma$ and~$\varepsilon_i=0$ otherwise.
    To $ \hatx\in \tref$ is associated its barycentric coordinates $\lambda_i$ associated to the vertices~$\hatv_i$ of~$\tref$ and $\lambdaetoile (\hatx):= \sum_{i=1}^{d+1} \varepsilon_i \lambda_i$ (shortly denoted by $\lambdaetoile$). Finally, we define~$\hat{\sigma} : = \{ \hatx \in \tref; \lambdaetoile(\hatx) = 0 \}$ and the function~$
    \haty := \dfrac{1}{\lambdaetoile}\sum_{i=1}^{d+1} \varepsilon_i \lambda_i\hatv_i\in\tref$, which is well defined on~$\trefminissigma$.
\end{definition}
%
%
%
%
%
\begin{figure}[h]
\centering
\begin{tikzpicture}[scale = 0.8]

\draw (0.5,0.5) node[above] {$\hat{T}$};

\draw (0,0) node  {$\bullet$};
\draw (2,0) node  {$\bullet$};
\draw (0,2) node  {$\bullet$};

\draw (0,0) node[below] {$\hatv_1$};
\draw (2,0) node[below]  {$\hatv_2$};
\draw (0,2) node[above]  {$\hatv_3$};

\draw (0,0) -- (2,0) ;
\draw (0,0) -- (0,2);
\draw[blue] (2,0) -- (0,2) ;








\draw[blue] (1.3,0.7) node  {$\bullet$};
\draw[blue] (0.7,0.3) node  {$\bullet$};

\draw[blue] (1.3,0.7) node[below] {$ \haty$};
\draw[blue] (0.7,0.3) node[left] {$ \hatx$};

\draw[blue] (1.3,0.7) -- (0.7,0.3) ;


\draw[blue] [->] (2.8,1) -- (4.3,1);
\draw[blue] (3.5,1) node[above] {${\ft}$};




\draw (7.4,0.8) node[above] {$T$};


\draw (6,0) node  {$\bullet$};
\draw (7.36,2.5) node  {$\bullet$};
\draw (9,0) node  {$\bullet$}; 
\draw (7.36,2.56) -- (9,0);
\draw[red] (7.36,2.56) -- (6,0);
\draw (9,0) -- (6,0);


\draw (6,0)  node[left] {$v_2$};
\draw (7.36,2.6) node[above]  {$v_3$};
\draw (9,0) node[below]  {$v_1$};





\draw plot [domain=-0.2:1.5] (\x+6,-1.786*\x^2+4.3*\x);

\draw (5,-0.6) node {$\Gamma$};
\draw[thick]   [->] (5.2,-0.6)--(5.88,-0.5);






\draw[red] (6.55,1) node  {$\bullet$};
\draw[red] (7.5,0.3) node  {$\bullet$};

\draw[red] (6.5,1) node[above] {${y}$};
\draw[red] (7.5,0.3) node[right] {${x}$};

\draw[red] (6.55,1) -- (7.5,0.3) ;



\end{tikzpicture}
\caption{Visualisation of the two functions
$
\hat{y}:  \, \hat{T}  \mapsto \hat{T}
$
and
$y: \, T  \mapsto \partial T \cap \Gamma
$
in definition \ref{def:fte-y} in a 2D case}
\label{fig:haty}
\end{figure}
Consider a non-internal mesh element $T \in \tauh$, having at least 2 vertices on the boundary, and the affine transformation $\ft$. In the two dimensional case, $\ft(\hatsigma)$ will consist of the only vertex of~$T$ that is not on the boundary $\Gamma$. In the three dimensional case, the tetrahedral~$T$ either has~2 or 3 vertices on the boundary. In the first case, $\ft(\hatsigma)$ is the edge of~$T$ joining its two internal vertices. In the second case, $\ft(\hatsigma)$ is the only vertex of $T$. 
\begin{definition}
\label{def:fte-y}
We denote~$\taue$ the mesh consisting of all exact elements $\te=\fte(\tref)$, where~$\fte = \ft$ for all internal elements, as for the case of non-internal elements $\fte$ is given by, 
\begin{equation}
  \label{eq:def-fte}
\fonction{\fte}{\tref  }{\te :=\fte( \tref) }{  \hatx}{\displaystyle  \fte( \hatx) := \left\lbrace 
\begin{array}{ll}
 x & {\rm if } \, \hatx \in \hatsigma, \\
     x+(\lambdaetoile)^{r+2} ( b(y) - y) &  {\rm if } \, \hatx \in \trefminissigma ,
\end{array}
\right.}
\end{equation}
with $x = \ft( \hatx)$ and $y = \ft( \haty)$. It has been proven in \cite{elliott} that $\fte$ is a $\c1$-diffeomorphism and~$\c{r+1}$ regular on $\tref$.  
\end{definition}
%
%
%
%
%
%
%
%
\begin{remark}
\label{rem:Fe_T}
For $x\in T \cap \ghh$, we have that $\lambdaetoile = 1$ 
and so $y = x$ inducing that~$\fte( \hatx) = b(x)$. Then
$\fte \circ \ft^{-1} = b$ on $T \cap \ghh$. 
\end{remark}

\subsection{Curved mesh $\taur$ of order $r$.}
\label{sub-sec:ftr}
The exact mapping $\fte$, defined in \eqref{eq:def-fte}, is interpolated as a polynomial of order $r \ge 1$ in the classical $\P r$-Lagrange basis on $\tref$. The interpolant is denoted by $\ftr$, which is a $\c1$-diffeomorphism and is in~$\c {r+1}(\tref)$ (see \cite[chap. 4.3]{PHcia}). For more exhaustive details and properties of this transformation, we refer to \cite{elliott,ciaravtransf,PHcia}. 
%
%
%
%
%
%
%
%
%
%
%
%
%
%
Note that, by definition, $\ftr$ and $\fte$ coincide on all $\P r$-Lagrange nodes. 
The curved mesh of order $r$ is $\taur := \{ \tr; T \in \tauh \}$, $\omhr := \cup_{\tr \in  \taur}\tr$ is the mesh domain and $\ghr:= \partial \omhr$ is its boundary. 

\section{Functional lift}
\label{section-lift}
%

We recall that $r \ge 1$ is the geometrical order of the curved mesh. With the help of aforementioned transformations, we define \textit{lifts} to transform a function on a domain $\omhr$ or $\ghr$ into a function defined on $\Omega$ or $\Gamma$ respectively, in order to compare the numerical solutions to the exact one. 

We recall that the idea of lifting a function from the discrete domain onto the continuous one was already treated and discussed in many articles dating back to the 1970's, like \cite{nedelec,scott-2,Lenoir1986,Bernardi1989} and others. Surface lifts were firstly introduced in 1988 by Dziuk in \cite{Dz88}, to the extend of our knowledge, and discussed in more details and applications by Demlow in many of his articles (see \cite{D1,D2,D3,D4}).

%
%
%
%
%
%
\subsection{Surface and volume lift definitions}
\label{sec:lift-def-surf-vol}
%
%
%
%
%
%
\begin{definition}[Surface lift]
\label{def:liftsurface}
    Let $u_h\in {\rm L}^2(\ghr)$. The surface lift $u_h^L\in {\rm L}^2(\Gamma)$ associated to $u_h$ is defined by, 
$$
        u_h^L\circ b := u_h,
$$
where $b: \ghr\rightarrow \Gamma$ is the orthogonal projection, defined in Proposition~\ref{tub_neigh_orth_proj_prop}. Likewise, to $u \in \L2(\Gamma)$ is associated its inverse lift $u^{-L}$ given by,
$
        u^{-L} := u \circ b \in \L2(\ghr).
$
\end{definition}
%
%
%
%
%
%
%
%
%
%
%
%
The use of the orthogonal projection $b$ to define the surface lift is 
natural since $b$ is well defined on the tubular neighborhood $\mathcal{U}_{\Gamma}$ of $\Gamma$ (see Proposition~\ref{tub_neigh_orth_proj_prop}) and henceforth on $\ghr \subset \mathcal{U}_{\Gamma}$ for sufficiently small mesh size $h$. 

A volume lift is defined, using the notations in definition \ref{def:sigma-lambdaetoile-haty}, we introduce the transformation~$\Ghr:~\omhr \rightarrow \Omega$ (see figure \ref{fig:Gh2}) given piecewise for all $\tr\in \taur$ by,
\begin{equation}
  \label{eq:def-fter}
  {\Ghr}_{|_{\tr}} \!\! := \ftre \circ ({\ftr})^{-1},
  \;
  \ftre( \hatx) \!:=\! \left\lbrace 
\begin{array}{ll}
 \!\!   x & {\rm if } \, \hatx \in \hatsigma \\
  \!\!   x+(\lambdaetoile)^{r+2} ( b(y) - y) &  {\rm if } \, \hatx \in \trefminissigma
\end{array}
\right.,
\end{equation}
with $ x := \ftr( \hatx)$ and $y := \ftr( \haty)$ (see figure~\ref{fig:haty} for the affine case). Notice that this implies that~${\Ghr}_{|_{\tr}} = id_{|_{\tr}}$, for any internal mesh element $\tr~\in~\taur$. Note that, by construction,~$\Ghr$ is globally continuous and piecewise differentiable on each mesh element. For the remainder of this article, the following notations are crucial. $\diff\Ghr$ denotes the differential of~$\Ghr$,~$(\diff{\Ghr})^t$ is its transpose and $\Jh$ is its Jacobin.

\begin{figure}[h]
\centering
\begin{tikzpicture}[scale = 0.8]





\draw[red] (0.4,0.8) node[above] {$\tdeux$};
\draw[red] (-1,0) arc (180:123:3);


\draw (-1,0) node  {$\bullet$};
\draw (0.36,2.5) node  {$\bullet$};
\draw (2,0) node  {$\bullet$}; 
\draw (0.36,2.56) -- (2,0);
\draw (2,0) -- (-1,0);




\draw (-1,0)  node[left] {$v_1$};
\draw (0.36,2.6) node[above]  {$v_3$};
\draw (2,0) node[below]  {$v_2$};

\draw (-0.7,1.35) node[left]  {$v_5$};
\draw (0.5,0) node[below]  {$v_4$};
\draw (1.15,1.3) node[right]  {$v_6$};

\draw[red] (0.5,0) node  {$\bullet$};
\draw[red] (-0.6,1.5) node  {$\bullet$}; 
\draw[red] (1.15,1.3) node  {$\bullet$};


\draw plot [domain=-0.3:1.5] (\x-1,-1.786*\x^2+4.3*\x);

\draw (-1.7,-1) node {$\Gamma$};


\draw[red] [->] (3.2,1) -- (4.8,1);
\draw[red] (4,1) node[above] {$\textcolor{red}{\Ghdeux}$};



\draw plot [domain=-0.3:1.5] (\x+6,-1.786*\x^2+4.3*\x);

\draw (5.3,-1) node {$\Gamma$};




\draw[red] (7.4,0.8) node[above] {$\te$};


\draw (6,0) node  {$\bullet$};
\draw (7.36,2.5) node  {$\bullet$};
\draw (9,0) node  {$\bullet$}; 
\draw (7.36,2.56) -- (9,0);
\draw (9,0) -- (6,0);




\draw (6,0)  node[left] {$v_1$};
\draw (7.36,2.6) node[above]  {$v_3$};
\draw (9,0) node[below]  {$v_2$};

\draw (6.3,1.3) node[left]  {$v_5$};
\draw (7.5,0) node[below]  {$v_4$};
\draw (8.15,1.3) node[right]  {$v_6$};

\draw[red] (7.5,0) node  {$\bullet$};
\draw[red] (6.4,1.5) node  {$\bullet$}; 
\draw[red] (8.15,1.3) node  {$\bullet$};
\end{tikzpicture}
\caption{Visualisation of $G_h^{(2)}: \tdeux \to \te$ in a 2D case, for a quadratic case $r=2$.}
\label{fig:Gh2}
\end{figure}
%
%
%
\begin{definition}[Volume lift]
\label{def:liftvolume}
    Let $u_h\in {\rm L}^2(\omhr)$. We define the volume lift associated to $u_h$, denoted $u_h^\ell \in {\rm L}^2(\Omega)$, by, 
    $$
        u_h^\ell\circ \Ghr := u_h.
    $$
    In a similar way, to $u\in {\rm L}^2(\Omega)$ is associated its inverse lift $u^{-\ell}\in \L2(\omhr)$ given by~$
    u^{-\ell} :=u \circ \Ghr.
    $
\end{definition}
%
%
%
%
%
%
%
\begin{proposition}
\label{prop:trace}
The volume and surface lifts coincide on $\ghr$,
\begin{displaymath}
  \forall ~ u_h \in {\rm H}^1(\omhr), \quad 
  \left ( {\rm Tr} ~u_h\right )^L = {\rm Tr} (u_h^\ell).
\end{displaymath}
Consequently, the surface lift $v_h^L$ (resp. the inverse lift~$v^{-L}$) will now be simply denoted by $v_h^\ell$ (resp.~$v^{-\ell}$). 
\end{proposition}
\begin{proof}
Taking $x\in \tr\cap\ghr$, $ \hatx = (\ftr)^{(-1)}(x)$ satisfies $\lambdaetoile=1$ and so~$\haty=\hatx$ and $y=x$. Thus~$\ftre( \hatx) = b(x) $, in other words, 
$$
    \Ghr(x)=\ftre\circ( \ftr)^{-1}(x) = b (x), \ \ \ \ \ \forall \ x \in \tr\cap \ghr .
$$
\end{proof}
%
%
%
%
%
%
%
%
%
\begin{proposition}
\label{prop-Gh}
Let $\tr \in \taur$. Then the mapping ${\Ghr}_{|_{\tr}}$ is $\c{r+1}(\tr)$ regular and a $\c 1$- diffeomorphism from $\tr$ onto $\te$. Additionally, for a sufficiently small mesh size $h$, there exists a constant $c>0$, independent of $h$, such that, 
\begin{equation}
\label{ineq:Gh-Id_Jh-1} 
  \forall \ x \in \tr, \ \ \ \ \| \diff {\Ghr}(x) - \I \| \le c h^r \qquad \mbox{ and } \qquad 
    | \Jh(x)- 1 | \le c h^r,
\end{equation}
where $\Ghr$ is defined in \eqref{eq:def-fter} and $\Jh$ is its Jacobin.
\end{proposition}
The full proof of this proposition is partially adapted from \cite{elliott} and has been detailed in appendix~\ref{appendix:proof-Ghr}.
\begin{remark}[Lift regularity]
\label{rem:lift-regularity}
The lift transformation $\Ghr:~\omhr \rightarrow \Omega$  in (\ref{eq:def-fter}) involves the function,
\begin{displaymath}
  \label{rem-def-rho}
  \rhotr :~ \hatx \in \tref \mapsto
 (\lambdaetoile)^{s} (b(y) - y),
\end{displaymath}
with an exponent $s=r+2$ inherited from \cite{elliott}: this exponent value guaranties the $\c{r+1}$ (piecewise) regularity of the function $\Ghr$.
However, decreasing that value to $s=2$  still ensures 
that $\Ghr$ is a (piecewise) $\c{1}$ diffeomorphism and also that 
Inequalities (\ref{ineq:Gh-Id_Jh-1}) hold: this can be seen when examining the proof of Proposition \ref{prop-Gh} in Appendix~\ref{appendix:proof-Ghr}.
Consequently,  the convergence theorem \ref{th-error-bound} still holds when setting~$s=2$ in the definition of $\rhotr$.
\end{remark}
\begin{remark}[Former lift definition]
\label{rem:lift-elliott-trace-non-conserve}
The volume lift defined in \eqref{def:liftvolume} is an adaptation of the lift definition in \cite{elliott}, which however does not fulfill Proposition \ref{prop:trace}. Precisely, in~\cite{elliott},  to~$u_h \in \HH1 (\omhr)$ is associated the lifted function \textcolor{blue}{$u_h^{e\ell}\in \HH1 (\Omega)$}, given by \textcolor{blue}{$u_h^{e\ell} \circ G_h := u_h$}, where 
$G_h:~ \omhr\rightarrow \Omega$ is defined piecewise, for each mesh element~$\tr \in \taur$, by 
$ {G_h}_{|_{\tr}} := \fte \circ ({\ftr})^{-1}$, 
where~$T$ is the affine element relative to $\tr$, $\fte$ is defined in \eqref{eq:def-fte} and $\ftr$ is its $\P r$-Lagrangian  interpolation given in section \ref{sub-sec:ftr}. 
%
%
%
%
%
%
However, this transformation does not coincide with the orthogonal projection~$b$, on the mesh boundary $\ghr$. 
Indeed, since $\fte \circ \ft^{-1} = b$ on~$T \cap \ghh$ (see Remark \ref{rem:Fe_T}), we have,
$$
    G_h(x) = b \circ \ft\circ ({\ftr})^{-1}(x) \ne b(x), \quad \forall \ x \in \ghr\cap\tr .
$$
Consequently in this case, $( {\rm Tr} ~u_h )^L \ne {\rm Tr} (\textcolor{blue}{u_h^{e\ell}})$. 
\end{remark}

%
%
%
%
%
%
%
%
%
%
%
%
%
%
%
%
%
%
%
%
%
%
\subsection{Lift of the variational formulation} 
\label{sub-sec:impact-lift}
With the lift operator, one may express an integral over $\ghr$ (resp.~$\omhr$) with respect to one over $\Gamma$ ( resp. $\Omega$), as will be discussed in this section. 

\paragraph{Surface integrals.}
In this subsection, all results stated may be found alongside their proofs in \cite{D1, demlow2019}, but we recall some necessary informations for the sake of completeness. For extensive details, we also refer to \cite{D2,actanum, Dz88}. Throughout the rest of the paper, $\, \d\sigma$ and $\, \d\sigma_h$ denote respectively the surface measures  on~$\Gamma$ and on $\ghr$. \medskip

Let $\Jb$ be the Jacobian of the orthogonal projection $b$, defined in Proposition~\ref{tub_neigh_orth_proj_prop}, such that $ \d\sigma(b(x))=\Jb(x) \d\sigma_h(x)$, for all $x \in \ghr$. Notice that $\Jb$ is bounded independently of $h$ and its detailed expression may be found in \cite{D1,D2}. Consider also the lift of $\Jb$ given by $\Jblifte \circ b = \Jb$ (see Definition~\ref{def:liftsurface}). \medskip

%
%
%
%
%
%
%
%
Let $u_h, v_h \in \HH1 (\ghh) $ with $u_h^\ell, v_h^\ell \in \HH1 (\Gamma)$ as their respected lifts. Then, one has,
\begin{equation}
    \label{pass_fct_scalaire_surface}
    \int_{\ghr}   u_h v_h \, \d\sigma_h = \int_{\Gamma}   u^{\ell }_h v^{\ell }_h  \frac{ \, \d\sigma}{\Jblifte}.
\end{equation}

A similar equation may be written with tangential gradients. We start by given the following notations. We denote the outer unit normal vector over~$\Gamma$ by $\nn$ and the outer unit normal vector over~~$\ghr= \partial \omhr$ by $\nnhr$. Denote $P:= \I-\nn \otimes \nn$ and~$P _h := \I-\nnhr \otimes \nnhr$ respectively as the orthogonal projections over the tangential spaces of $\Gamma$ and~$\ghr$. Additionally, the Weingarten map~$\hessienne: \R^{d } \to \R^{d\times d}$ is given by~$\hessienne :=\diff^2 \d $, where~$\d$ is the signed distance function (see Proposition~\ref{tub_neigh_orth_proj_prop}). 
With the previous notations, we have,
$$
    \nabla_{\ghh }{v_h(x)}=P _h(I-\d \hessienne)P\nabla_{\Gamma}{v^{\ell }_h(b(x))}, \  \ \ \ \ \forall \ x \in \ghr.
$$
%
%
%
%
%
%
%
%
%
%
Using this equality, we may derive 
the following expression, 
\begin{equation}
\label{pass_grad_surface}
    \int_{\ghr} \nabla_{\ghr} u_h \cdot \nabla_{\ghr}  v_h \, \d\sigma_h = \int_{\Gamma} \Ahell \nabla_{\Gamma} u^{\ell }_h \cdot \nabla_{\Gamma} v^{\ell }_h  \, \d\sigma,
\end{equation} 
where $\Ahell$ is the lift of the matrix $A_h$ given by,
\begin{equation}
\label{eq:expression_Ah}
   A_h(x) := \frac{1}{\Jb(x)}P(I-\d \hessienne)P _h(I-\d \hessienne)P(x), \ \ \ \ \ \forall \ x \in \ghr.
\end{equation}

\paragraph{Volume integrals.} 
%
%
%
%
%
%
%
%
Similarly, consider $u_h, v_h \in \HH1 (\omhh) $ and let $u_h^\ell, v_h^\ell \in \HH1 (\Omega)$ be their respected lifts (see Definition \ref{def:liftvolume}), we have,
\begin{equation}
\label{pass_fct_scalaire_volume}
    \int_{\omhh}u_h v_h \, \d x = \int_\Omega u_h^\ell v_h^\ell \frac{1}{\Jhlifte} dy,
\end{equation} 
where $\Jh$ denotes the Jacobian of $\Ghr$ and $\Jhlifte$ is its lift given by $\Jhlifte \circ \Ghr = \Jh$. 

Additionally, the gradient can be written as follows, for any $x \in \omhr$,
$$
    \nabla v_h(x) =\nabla( v_h^\ell \circ {\Ghr})(x)= {\transpose}\diff {\Ghr}(x) ( \nabla v_h^\ell)\circ{(\Ghr(x))}.
$$
Using a change of variables $z=\Ghr(x) \in \Omega$, one has,
$
    (\nabla v_h)^\ell(z) = {\transpose} \diff {\Ghr}(x)  \nabla v_h^\ell{(z)}. 
$
Finally, introducing the notation,
\begin{equation}
\label{eq:matrice-Ghr}
    \G (z) := {\transpose}\diff {\Ghr}(x),
\end{equation}
one has,
\begin{equation}
\label{pass_grad_volume}
     \int_{\Omega^{(r)}_h} \nabla u_h \cdot \nabla v_h \, \d x = \int_{\Omega}  \G (\nabla u_h^\ell) \cdot \G (\nabla v_h^\ell) \frac{\, \d x}{\Jhlifte}.
\end{equation}
\subsection{Useful estimations}
\paragraph{Surface estimations.}
We recall two important estimates proved in~\cite{D1}. There exists a constant~$c>0$ independent of $h$ such that,  
\begin{equation}
  \label{ineq:AhJh}
    ||\Ahell -P||_{\Linf(\Gamma)} \le c h^{r+1} \qquad \mbox{ and } \qquad 
        \left\| 1-\frac{1}{\Jblifte} \right\|_{\Linf(\Gamma)} \le ch^{r+1},
\end{equation}
where $\Ahell$ is the lift of $A_h$ defined in \eqref{eq:expression_Ah} and $\Jb$ is the Jacobin of the projection $b$.
%
%
%
%
%
%
%
%
\medskip

\paragraph{Volume estimations.}

    A direct consequence of the proposition \ref{prop-Gh} is that both $\mathrm{D}\Ghr$ and $\Jh$ are bounded on every $\tr \in \taur$. As an extension of that, by Definition~\ref{def:liftvolume} of the lift, both $\G$ and~$\Jhlifte$ are also bounded on $\te$. Additionally, the inequalities 
    \eqref{ineq:Gh-Id_Jh-1} will not be directly used in the error estimations in Section~\ref{ERROR-section}, the following inequalities will be used instead,
    \begin{equation}
        \label{ineq:Ghr-Id_1/Jh-1}
        \forall \ x \in \te, \ \ \ \ \| \G (x) - \I \| \le c h^r \qquad \mbox{ and } \qquad
         \left| \frac{1}{\Jhlifte(x)}- 1 \right| \le c h^r,
    \end{equation}
    where $\G$ is given in \eqref{eq:matrice-Ghr}. These inequalities are a consequence of the lift applied on the inequalities 
\eqref{ineq:Gh-Id_Jh-1}.
%
%
%
%
%
%
%
%
%
%
%
%
%
\begin{remark}
\label{rem:norm-equiv}
Let us emphasize that, there exists an equivalence between the $\Hexpo{m}$-norms over~$\omhh$ (resp. $\ghh$) and the $\Hexpo{m}$-norms over $\Omega$ (resp. $\Gamma$), for $m=0,1$. 
Let $v_h \in \HH1 (\omhh,\ghh) $ and let~$v_h^\ell \in \HH1 \omgam$ be its lift, then for~$m=0, 1$, there exist strictly positive constants independent of $h$ such that,
\begin{equation*}
\begin{array}{rcccl}
c_1 \|v_h^\ell\|_{\Hexpo{m}(\Omega)} &\le& \|v_h\|_{\Hexpo{m}(\omhh)} &\le&  c_2 \|v_h^\ell\|_{\Hexpo{m}(\Omega)}, \\
c_3 \| v_h^\ell\|_{\Hexpo{m}(\Gamma)} &\le& \| v_h\|_{\Hexpo{m}(\ghh)} &\le&  c_4 \| v_h^\ell\|_{\Hexpo{m}(\Gamma)}.
\end{array}
\end{equation*}
The second estimations are proved in \cite{D1}. As for the first inequalities, one may prove them while using the equations \eqref{pass_fct_scalaire_volume} and \eqref{pass_grad_volume}. They hold due to the fact that~$\Jh$ and~$\diff \Ghr$ (respectively~$\frac{1}{\Jhlifte}$ and~$\G$) are bounded on $\tr$ (resp. $\te$), as a consequence of the proposition~\ref{prop-Gh} and the inequalities in \eqref{ineq:Ghr-Id_1/Jh-1}. 

%
%
%
\end{remark}
%
%
%
%
%
\section{Finite element approximation}
\label{FE-section}
In this section, is presented the finite element approximation of problem \eqref{1} using $\P k$-Lagrange finite element approximation. We refer to \cite{EG,PHcia} for more details on finite element methods. 
%
%
%
%
%
\subsection{Finite element spaces and interpolant definition}
\label{subsec:fe-space-interpolant}
%
%
%
%
%
%
Let~$k \geq 1$, given a curved mesh $\taur$, the $\P k$-Lagrangian finite element space is given by, 




%
%
%
%
%
%
$$
\Vh := \{ \chi \in C^0(\omhr); \ \chi_{|_T}= \hat{\chi} \circ (\ftr )^{-1} , \ \hat{\chi} \in \mathbb{P}^k(\hat{T}), \ \forall \ T \in \taur \}.
$$
%
%
%
%
%
Let the $\P r$-Lagrangian interpolation operator be denoted by $\Ihr: v \in \c0 (\omhr) \mapsto \Ihr (v) \in \Vh$. 
%
%
%
%
%
%
The lifted finite element space (see Section \ref{sec:lift-def-surf-vol} for the lift definition), is defined by,
$$
     \Vhlifte := \{ v_h^\ell; \ v_h \in \Vh \},
$$
and its lifted interpolation operator $\Ihlifte$ given by, 
\begin{equation}
\label{def:interpolation-op-lifte}
    \fonction{\Ihlifte}{\c0 ({\Omega})}{\Vhlifte}{v}{\Ihlifte (v) := \big( \Ihr (v^{-\ell}) \big)^\ell.} 
\end{equation}
Notice that, since $\Omega$ is an open subset of $\R^2$ or $\R^3$, then we have the following Sobolev injection~$\Hexpo{k+1}(\Omega) \hookrightarrow \c0 (\Omega)$. 
Thus, any function $w \in \Hexpo{k+1}(\Omega)$ may be associated to an interpolation element~$\Ihlifte(w) \in \Vhlifte$. \medskip 

The lifted interpolation operator plays a part in the error estimation and the following interpolation inequality will display the finite element error in the estimations. 
\begin{proposition}
\label{prop:interpolation-ineq}
Let $v \in \Hk1 \omgam$ and $2 \le m \le k+1$. There exists a constant~$c>0$ independent of $h$ such that the interpolation operator $\Ihlifte$ satisfies the following inequality,
$$
    \|v-\Ihlifte v\|_{\L2\omgam} + h \|v-\Ihlifte v\|_{\HH1\omgam} \le c h^{m} \|v\|_{\Hexpo{m} \omgam}. 
$$
%
%
%
%
%
%
\end{proposition}
\begin{proof}
This inequality derives from given interpolation theory, see~\cite[Corollary~4.1]{Bernardi1989} and \cite{Brenner-scott} for norms over $\Omega$ and ~\cite{D1,D2} for norms over $\Gamma$. One also needs to use the following inequality, $ \|v^{-\ell}\|_{\Hexpo{m}(\tr)} \le  c \|v\|_{\Hexpo{m}(\te)},$ for $0 \le  m \le k+1,$ where the constant $c$ is independent of $h$. This inequality follows from  a change of variables and the fact that  $\diff^m \Ghr = \I + \diff^m (\rhotr \circ (\ftr)^{-1}) $ is locally bounded independently of $h$, which is easily proved using \cite[page 19]{ciaravtransf} and~\eqref{ineq:rho_tr}. 
\end{proof}

\subsection{Finite element formulation}
From now on, to simplify the notations, we denote~$\omhh$ and $\ghh$ to refer to~$\omhr$ and~$\ghr$, for any geometrical order $r\ge 1$. 

\paragraph{Discrete formulation.}
Given $f \in \L2(\Omega)$ and $g \in \L2(\Gamma)$ the right hand side of Problem~\eqref{1}, we define (following \cite{elliott,D1}) the following linear form $l_h$ on $\Vh$ by,
$$
    l_h(v_h) := \int_{\Omega _h} v_h f^{-\ell}\Jh  \, \d x + \int_{\ghh} v_h g^{-\ell}\Jb \, \d\sigma_h,
$$
where $\Jh$ (resp. $\Jb$) is the Jacobin of $\Ghr$ (resp. the orthogonal projection $b$). 
With this definition,~$l_h(v_h)=l(v_h^\ell)$, for any $v_h \in \Vh$, where $l$ is the right hand side in the formulation \eqref{fv_faible}. 

The approximation problem is to find~$u_h \in \mathbb{V}_h$ such that, 
\begin{equation}
\label{fvh}
     a_h(u_h,v_h) = l_h(v_h), \ \ \ \ \ \forall \ v_h \in \Vh,
\end{equation}
where $a_h$ is the following bilinear form, defined on $\Vh \times \Vh$,
\begin{align*}
    a_h(u_h,v_h) &:= \int_{\Omega _h} \nabla u_h \cdot \nabla v_h \, \d x + \kappa \int_{\Omega _h} u_h  v_h \, \d x  +\beta \int_{\ghh} \nabla_{\ghh} u_h \cdot \nabla_{\ghh} v_h \, \d\sigma_h 
     + \alpha \int_{\ghh} u_h v_h \, \d\sigma_h.
\end{align*}
\begin{remark}
Since $a_h$ is bilinear symmetric positively defined on a finite dimensional space, then there exists a unique solution $u_h \in \Vh$ to the discrete problem~\eqref{fvh}. 

\end{remark}

\paragraph{Lifted discrete formulation.}
We define the lifted bilinear form $\ahell$, defined on $\Vhlifte \times \Vhlifte$, throughout, $$a^\ell_h(u^\ell_h,v^\ell_h)=a_h(u_h,v_h) \quad  \mbox{ for } u_h, v_h \in \Vh,$$ 
applying \eqref{pass_grad_volume}, \eqref{pass_fct_scalaire_volume}, \eqref{pass_grad_surface} and~\eqref{pass_fct_scalaire_surface}, its expression is given by, 
\begin{multline*}
    \ahell(u^\ell_h,v^\ell_h) = \int_{\Omega}  \G (\nabla u_h^\ell) \cdot \G (\nabla v_h^\ell) \frac{\, \d x}{\Jhlifte}+\beta \int_{\Gamma} A^{\ell }_h \nabla_{\Gamma} u^{\ell }_h \cdot \nabla_{\Gamma} v^{\ell }_h  \, \d\sigma \\
     + \kappa \int_{\Omega} (u_h)^\ell (v_h)^\ell \frac{\, \d x}{\Jhlifte}  + \alpha \int_{\Gamma} (u_h)^\ell (v_h)^\ell \frac{\, \d\sigma}{\Jblifte}.
\end{multline*}

Keeping in mind that $u$ is the solution of \eqref{fv_faible} and $u_h^\ell$ is the lift of the solution of \eqref{fvh}, for any $v_h^\ell \in \Vhlifte \subset \HH1 \omgam$, we notice that,
\begin{equation}
\label{rem:a=ahell-For-vhell}
  a(u,v_h^\ell) = l(v_h^\ell)=  l_h(v_h)= a_h(u_h,v_h) = a^\ell_h(u^\ell_h,v^\ell_h).
\end{equation} 
Using the previous points, we can also define the lifted formulation of the discrete problem \eqref{fvh} by: find $u_h^\ell \in \mathbb{V}_h^\ell$ such that,
$$
    \ahell(u^\ell_h,v^\ell_h)= l(v_h^\ell), \ \  \ \ \ \forall \ v_h^\ell \ \in \mathbb{V}_h^\ell.
$$
%
%
%
%
%
%
\section{Error analysis}
\label{ERROR-section}
Throughout this section, we consider that the mesh size~$h$ is sufficiently small and that $c$ refers to a positive constant independent of the mesh size~$h$. From now on, the domain~$\Omega$, is assumed to be at least~$\c {k+1}$ regular, and the source terms in problem \eqref{1}  are assumed more regular: $f \in \mathrm{H}^{k-1}(\Omega)$ and $g \in \mathrm{H}^{k-1}(\Gamma)$. Then according to~\cite[Theorem~3.4]{ventcel1}, the exact solution $u$ of Problem~\eqref{1} is in~$\Hk1\omgam$. \medskip

Our goal in this section is to prove the following theorem. 
%
%
%
%
\begin{theorem}
\label{th-error-bound}
Let $u \in \Hk1\omgam$ be the solution of the variational problem \eqref{fv_faible} and $u_h \in \Vh$ be the solution of the finite element formulation~\eqref{fvh}. There exists a constant $c > 0$ such that for a sufficiently small mesh size $h$,
\begin{equation}
\label{errh1_errl2}
    \|u-u_h^\ell \|_{\HH1(\Omega, \Gamma) } \le c ( h^k + h^{r+1/2})
\quad \mbox{ and } \quad
    \|u-u_h^\ell \|_{\L2(\Omega, \Gamma) } \le c ( h^{k+1} + h^{r+1}),
\end{equation}
%
%
%
where $u_h^\ell \in \Vhlifte$ denotes the lift of $u_h$ onto $\Omega$, given in Definition \ref{def:liftvolume}.
\end{theorem}
%
%
%
%
%
The overall error in this theorem is composed of two components: the geometrical error and the finite element error. To prove these error bounds, we proceed as follows:
\begin{enumerate}
    \item estimate the geometric error: we bound the difference between the exact bilinear form $a$ and the lifted bilinear form $\ahell$;
    \item bound the $\HH1$ error using the geometric and interpolation error estimation, proving the first inequality of \eqref{errh1_errl2};
    \item 
    an Aubin-Nitsche argument helps us prove the second inequality of \eqref{errh1_errl2}.
\end{enumerate}
%
%
%
%
%
\subsection{Geometric error}
First of all, we introduce $B_h^\ell \subset \Omega$ as the union of all the non-internal elements of the exact mesh $\taue$, 
$$
    B_h^\ell=  \{ \ \te \in \taue ; \  \te \, \mbox{has at least two vertices on } \Gamma \}. 
$$
Note that, by definition of $B_h^\ell$, we have,
\begin{equation}
    \label{JDG}
    \frac{1}{\Jhlifte}-1 =0  \ \ \ \mbox{and} \ \ \ \G - \I = 0 \ \ \ \mbox{in} \ \Omega \backslash B_h^\ell.
\end{equation}
The following corollary involving  $B_h^\ell$ is a direct consequence of \cite[Lemma 4.10]{elliott} or~\cite[Theorem~1.5.1.10]{Grisvard2011}.
\begin{corollary}
Let $v \in \HH1(\Omega)$ and $w \in \H2(\Omega)$. Then, for a sufficiently small~$h$, there exists $c>0$ such that the following inequalities hold,
\begin{equation}
\label{h1/2_blh}
    \|v\|_{\L2(B_h^\ell)} \le c h^{1/2} \|v\|_{\HH1(\Omega)} 
    \qquad \mbox{and} \qquad
    \|w\|_{\HH1(B_h^\ell)} \le c h^{1/2} \|w\|_{\H2(\Omega)}.
\end{equation}
\end{corollary}
%
%
%
%
%
%
%
%
%
%
%
%
%
%
The difference between $a$ and $a_h$, referred to as the geometric error, is evaluated in the following proposition.
\begin{proposition}
Consider $v, w \in \mathbb{V}_h^\ell$. Then for a sufficiently small~$h$, there exists $c>0$, such that the following geometric error estimation hold,
\begin{equation}
    \label{ineq:a-al2}
    |a(v,w)-\ahell(v,w)| \le c h^r \|\na v\|_{\L2(B_h^\ell)}\|\na w\|_{\L2(B_h^\ell)} + ch^{r+1}\|v\|_{\HH1 \omgam}\|w\|_{\HH1 \omgam}. 
    \end{equation}
\end{proposition}

The following proof is inspired by \cite[Lemma 6.2]{elliott}. The main difference is the use of the modified lift given in definition \ref{def:liftvolume} and the corresponding transformation $\Ghr$ alongside its associated matrix~$\G$, defined in \eqref{eq:matrice-Ghr}, which leads to several changes in the proof. 

\begin{proof}
Let $v, w \in \mathbb{V}_h^\ell$. 
By the definitions of the bilinear forms $a$ and $\ahell$, we have,
$$
     |a(v,w)-\ahell(v,w)| \le a_1(v,w) +\kappa a_2(v,w) +\beta a_3(v,w)+ \alpha a_4(v,w),
$$
where the terms $a_i$, defined on ${\Vhlifte}\times\Vhlifte$, are respectively given by,
%
\begin{equation*}
\begin{array}{rclrcl}
     \!\!a_1(v,w) &\!\!\!\!\!:=\!\!\!\!\!& \displaystyle \left|\int_{\Omega} \nabla w \cdot \nabla v - \G \nabla w \cdot \G \nabla v \frac{1}{\Jhlifte}\, \d x\right|, & 
     \!\!a_2(v,w) &\!\!\!\!\!:=\!\!\!\!\!& \displaystyle \left|\int_{\Omega}  w  v \ (1- \frac{1}{\Jhlifte}) \, \d x\right|, \!\!\\
    \!\!a_3(v,w) &\!\!\!\!\!:=\!\!\!\!\!& \displaystyle \left|\int_{\Gamma} (\Ahell - \I) \ \nabla_{\Gamma} w \cdot \nabla_{\Gamma} v \, \d\sigma\right|, & 
     \!\!a_4(v,w) &\!\!\!\!\!:=\!\!\!\!\!& \displaystyle  \left|\int_{\Gamma} w  v \ (1- \frac{1}{\Jblifte})\, \d\sigma\right|.\!\!
\end{array}
\end{equation*}
The next step is to bound each~$a_i$, for $i=1 ,2 ,3, 4$, while using  
\eqref{ineq:Ghr-Id_1/Jh-1}
and \eqref{ineq:AhJh}. \medskip

First of all, notice that 
$
    a_1(v,w) \le Q_1 +Q_2+Q_3
$, 
where, 
\begin{eqnarray*}
     Q_1 &:=& \displaystyle \left|\int_{\Omega} ( \G -\I) \  \nabla w \cdot \G \nabla v \frac{1}{\Jhlifte}\, \d x\right|,\\
     Q_2 &:=& \displaystyle \left|\int_{\Omega} \nabla w \cdot ( \G -\I)\nabla v \frac{1}{\Jhlifte}\, \d x\right|,\\
     Q_3 &:=& \displaystyle \left|\int_{\Omega} \nabla w \cdot \nabla v (\frac{1}{\Jhlifte}-1) \, \d x\right|.
\end{eqnarray*}
We use \eqref{JDG}
and \eqref{ineq:Ghr-Id_1/Jh-1} 
to estimate each $Q_j$ as follows,
\begin{align*}
    & Q_1 = \left|\int_{B_h^\ell} ( \G -\I) \  \nabla w \cdot \G \nabla v \frac{1}{\Jhlifte}\, \d x\right|\le ch^r \| \nabla w\|_{\L2(B_h^\ell)} \| \nabla v\|_{\L2(B_h^\ell)},\\
    & Q_2 = \left|\int_{B_h^\ell} \nabla w \cdot ( \G -\I)\nabla v \frac{1}{\Jhlifte}\, \d x\right| \le ch^r \| \nabla w\|_{\L2(B_h^\ell)} \| \nabla v\|_{\L2(B_h^\ell)},\\
    & Q_3 = \left|\int_{B_h^\ell} \nabla w \cdot \nabla v (\frac{1}{\Jhlifte}-1) \, \d x \right| \le ch^r \| \nabla w\|_{\L2(B_h^\ell)} \| \nabla v\|_{\L2(B_h^\ell)}.
\end{align*}
Summing up the latter terms, we get,
$a_1(v,w) \le ch^r \| \nabla w\|_{\L2(B_h^\ell)} \| \nabla v\|_{\L2(B_h^\ell)}.$ \medskip

Similarly, to bound $a_2$, we proceed by using \eqref{JDG} and 
\eqref{ineq:Ghr-Id_1/Jh-1} 
as follows,
\begin{align*}
    a_2(v,w)&  = \left|\int_{B_h^\ell}  w  v \ (1- \frac{1}{\Jhlifte}) \, \d x\right|  \le c h^r \|  w\|_{\L2(B_h^\ell)}\|  v\|_{\L2(B_h^\ell)}.
\end{align*}
Since $v,w \in~\mathbb{V}_h^\ell \subset \HH1\omgam$, we use 
\eqref{h1/2_blh} 
to get,
$$
    a_2(v,w) \le c h^{r+1} \|w\|_{\HH1(\Omega)}\|  v\|_{\HH1(\Omega)}.
$$

Before estimating $a_3$, we need to notice that, by definition of the tangential gradient over~$\Gamma$, $P \nt = \nt $ 
where~$P=\I-\nn \otimes \nn$ is the orthogonal projection over the tangential spaces of $\Gamma$. 
With the estimate 
\eqref{ineq:AhJh}, we get,
\begin{align*}
   a_3(v,w)  
     & = \left| \int_{\Gamma} (\Ahell - P) \ \nabla_{\Gamma} w \cdot \nabla_{\Gamma} v \, \d\sigma \right|  \\
   & \le ||\Ahell -P||_{\Linf(\Gamma)} \| w\|_{\HH1(\Gamma)} \| v\|_{\HH1(\Gamma)} 
     \le ch^{r+1} \| w\|_{\HH1(\Gamma)} \| v\|_{\HH1(\Gamma)} .
\end{align*}

Finally, using 
\eqref{ineq:AhJh}, we estimate $a_4$ as follows,
\begin{align*}
   a_4(v,w) = \left| \int_{\Gamma} w  v \ (1- \frac{1}{\Jblifte})\, \d\sigma\right|  \le ch^{r+1} \|  w\|_{\L2(\Gamma)}\|  v\|_{\L2(\Gamma)}.
\end{align*}

The inequality \eqref{ineq:a-al2} is easy to obtain when summing up $a_i$, for all $i=1, 2, 3, 4$.
\end{proof}
%
%
%
%
%
%
%
%
%
%
%
%
%
%
%
%
%
\begin{remark}
Let us point out that, with $u$ (resp. $u_h$) the solution of the problem~\eqref{fv_faible} (resp. \eqref{fvh}), we have, 
\begin{equation}
    \label{rem:bound-uhell-indep-h}
     \|u_h^\ell \|_{\HH1 \omgam} \le c\|u \|_{\HH1 \omgam},
\end{equation}
where $c>0$ is independent with respect to $h$.
In fact, a relatively easy way to prove it is by employing the geometrical error estimation \eqref{ineq:a-al2}, as follows,
\begin{align*}
   c_c \|u_h^\ell \|_{\HH1 \omgam}^2  \le a(u_h^\ell,u_h^\ell)
   \le a(u_h^\ell,u_h^\ell)-a(u,u_h^\ell) + a(u,u_h^\ell),
\end{align*}
where $c_c$ is the coercivity constant. Using \eqref{rem:a=ahell-For-vhell}, we have,
\begin{equation*}
    c_c \|u_h^\ell \|_{\HH1 \omgam}^2 \le a(u_h^\ell,u_h^\ell)-a_h^\ell(u_h^\ell,u_h^\ell) + a(u,u_h^\ell)
    = (a-a_h^\ell)(u_h^\ell,u_h^\ell) + a(u,u_h^\ell).
\end{equation*}
Thus applying the estimation \eqref{ineq:a-al2} along with the continuity of $a$, we get,
\begin{align*}
    \|u_h^\ell \|_{\HH1 \omgam}^2  &\le c h^r \|\na u_h^\ell\|^2_{\L2(B_h^\ell)}+ ch^{r+1}\|u_h^\ell\|_{\HH1 \omgam}^2+ c\|u \|_{\HH1 \omgam} \| u_h^\ell\|_{\HH1 \omgam} \\
    & \le c h^r \|u_h^\ell\|^2_{\HH1 \omgam}+ c\|u \|_{\HH1 \omgam} \| u_h^\ell\|_{\HH1 \omgam}.
\end{align*}
Thus, we have, 
$$
   (1-ch^r) \|u_h^\ell \|_{\HH1 \omgam}^2
    \le c\|u \|_{\HH1 \omgam} \| u_h^\ell\|_{\HH1 \omgam}.
$$
For a sufficiently small $h$, we have $1-ch^r>0$, which concludes the proof.
\end{remark}
\subsection{Proof of the $\HH1$ error bound in Theorem \ref{th-error-bound}}

Let $u \in \Hk1 \omgam$ and~$u_h \in \Vh$ be the respective solutions of \eqref{fv_faible} and \eqref{fvh}. 

To begin with, 
we use the coercivity of the bilinear form $a$ to obtain, denoting~$c_c$ as the coercivity constant,
\begin{multline*}
    c_c \|\Ihlifte u-u_h^\ell\|_{\HH1\omgam}^2 \le a(\Ihlifte u-u_h^\ell,\Ihlifte u-u_h^\ell) 
    = a(\Ihlifte u,\Ihlifte u-u_h^\ell)-a(u_h^\ell,\Ihlifte u-u_h^\ell)  \\
     = \ahell(u_h^\ell,\Ihlifte u-u_h^\ell)-a(u_h^\ell,\Ihlifte u-u_h^\ell) +a(\Ihlifte u,\Ihlifte u-u_h^\ell) 
     - \ahell(u_h^\ell,\Ihlifte u-u_h^\ell),
\end{multline*}
where in the latter equation, we added and subtracted $\ahell(u_h^\ell,\Ihlifte u-u_h^\ell)$. Thus,
\begin{equation*}
   c_c \|\Ihlifte u-u_h^\ell\|_{\HH1\omgam}^2 \le \big( \ahell-a\big)(u_h^\ell,\Ihlifte u-u_h^\ell) +a(\Ihlifte u,\Ihlifte u-u_h^\ell) - \ahell(u_h^\ell,\Ihlifte u-u_h^\ell).
\end{equation*}
Applying \eqref{rem:a=ahell-For-vhell} with $v=\Ihlifte u -u_h^\ell \in \Vhlifte$, we have,
$$
  c_c \|\Ihlifte u-u_h^\ell\|_{\HH1\omgam}^2 \le | (\ahell-a)(u_h^\ell,\Ihlifte u-u_h^\ell)| + |a(\Ihlifte u -u ,\Ihlifte u-u_h^\ell)|.
$$
Taking advantage of the continuity of $a$ 
and the estimate \eqref{ineq:a-al2}, we obtain,
\begin{equation*}
\begin{array}{l}
 c_c \|\Ihlifte u-u_h^\ell\|_{\HH1\omgam}^2 \\[0.1cm]
 \begin{array}{rcl}
 & \le & \!\!\!\displaystyle c \big(  h^r \|\na u_h^\ell\|_{\L2(B_h^\ell)}\|\na (\Ihlifte u-u_h^\ell)\|_{\L2(B_h^\ell)}  + h^{r+1}\|u_h^\ell\|_{\HH1\omgam}\|\Ihlifte u-u_h^\ell\|_{\HH1\omgam}\big) \!\!\! \\[0.1cm]
  & & \displaystyle  \qquad{ }  + c_{cont} \|\Ihlifte u -u\|_{\HH1\omgam} \|\Ihlifte u-u_h^\ell\|_{\HH1\omgam} \\[0.1cm]
  & \le &\!\!\! \displaystyle  c \big( h^r \|\na u_h^\ell\|_{\L2(B_h^\ell)} + h^{r+1}\|u_h^\ell\|_{\HH1\omgam}  \\[0.1cm]
  & & \displaystyle  \qquad{ } + c_{cont} \|\Ihlifte u -u\|_{\HH1\omgam} \big) \|\Ihlifte u-u_h^\ell\|_{\HH1\omgam}.
\end{array}
\end{array}
\end{equation*}
%
Then, dividing by $\|\Ihlifte u-u_h^\ell\|_{\HH1\omgam}$, we have,
\begin{align*}
    \|\Ihlifte u-u_h^\ell\|_{\HH1\omgam} & \le c\left( h^r \|\na u_h^\ell\|_{\L2(B_h^\ell)} +  h^{r+1}\|u_h^\ell\|_{\HH1\omgam} + \|\Ihlifte u -u\|_{\HH1\omgam} \right).
\end{align*}
To conclude, we use the latter inequality in the following estimate as follows,
\begin{equation*}
\begin{array}{l}
     \| u-u_h^\ell\|_{\HH1\omgam} 
     \le \| u-\Ihlifte u\|_{\HH1\omgam} + \|\Ihlifte u-u_h^\ell\|_{\HH1\omgam} \\[0.1cm]
       \le  \displaystyle c\left( h^r \|\na u_h^\ell\|_{\L2(B_h^\ell)} +  h^{r+1}\|u_h^\ell\|_{\HH1\omgam} + \|\Ihlifte u -u\|_{\HH1\omgam} \right)  
\end{array} 
\end{equation*}
Using the proposition \ref{prop:interpolation-ineq} and the inequalities \eqref{h1/2_blh}, we have,
\begin{equation*}
\begin{array}{l}
\| u-u_h^\ell\|_{\HH1\omgam} \\[0.1cm]
      \le  \displaystyle c h^r (\|\na (u_h^\ell-u)\|_{\L2(B_h^\ell)}+\|\na u\|_{\L2(B_h^\ell)}) + c h^{r+1}\|u_h^\ell\|_{\HH1\omgam} +   c h^k \|u\|_{\Hk1\omgam}\\[0.1cm]
      \le  \displaystyle c h^r (\|u_h^\ell-u\|_{\HH1 \omgam}+  h^{1/2}\| u\|_{\H2(\Omega)}) + ch^{r+1}\|u_h^\ell\|_{\HH1\omgam}  + c h^k \|u\|_{\Hk1\omgam}.
\end{array}
\end{equation*}
Thus we have,
\begin{align*}
    (1-c h^r) \| u-u_h^\ell\|_{\HH1\omgam} & \le c \left( h^{r+1/2}\| u\|_{\H2(\Omega)} + h^k \|u\|_{\Hk1\omgam} + h^{r+1}\|u_h^\ell\|_{\HH1\omgam}\right). 
\end{align*}
For a sufficiently small $h$, we arrive at,
\begin{align*}
    \| u-u_h^\ell\|_{\HH1\omgam} & \le c \left( h^{r+1/2}\| u\|_{\H2\omgam} + h^k \|u\|_{\Hk1\omgam} + h^{r+1}\|u_h^\ell\|_{\HH1\omgam} \right).
\end{align*}
This provides the desired result using \eqref{rem:bound-uhell-indep-h}.

%
%
%
%
%
%
\subsection{Proof of the $\L2$ error bound in Theorem \ref{th-error-bound}}

Recall that $u \in \HH1\omgam$ is the solution of the variational problem \eqref{fv_faible},~$u_h \in \mathbb{V}_h$ is the solution of the discrete problem \eqref{fvh}. 
To estimate the $\L2$ norm of the error, we define the functional~$F_h$ by,
$$
\fonction{F_h }{\HH1 \omgam  }{\R}{v}{\displaystyle   F_h(v)=a(u-u_h^\ell,v).}
$$
We bound $|F_h(v)|$ for any $v \in \H2 \omgam$ in Lemma \ref{lem:Fh}. Afterwards an Aubin-Nitsche argument is applied to bound the $\L2$ norm of the error.
\begin{lem}
\label{lem:Fh}
For all $v \in \H2\omgam$ {and for a sufficiently small~$h$,} there exists $c>0$ such that the following inequality holds,
\begin{equation}
\label{ineq:F-h}
    |F_h(v)| \le c ( h^{k+1} + h^{r+1} ) \|v\|_{\H2\omgam}.
\end{equation}
\end{lem}
\begin{remark}
To prove Lemma \ref{lem:Fh}, some key points for a function $v \in \H2 \omgam$ are presented. Firstly, inequality 
\eqref{h1/2_blh} implies that, 
\begin{equation}
\label{ineq:blh-for-v}
   \forall \, v \in \H2 \omgam, \quad \|\na v\|_{\L2(B_h^\ell)} \le c h^{1/2} \|v\|_{\H2(\Omega)}.
\end{equation}
Secondly, then the interpolation inequality in proposition \ref{prop:interpolation-ineq} gives, 
\begin{equation}
    \label{ineq:interpolation-v}
   \forall \, v \in \H2 \omgam, \quad  \| \Ihlifte v - v \|_{\HH1 \omgam} \le c h \|v\|_{\H2 \omgam}.
\end{equation}
Applying \ref{rem:a=ahell-For-vhell} for $\Ihlifte v \in \Vhlifte$, we have, 
\begin{equation}
    \label{eq:a=ahell-For-Iv}
   \forall \, v \in \H2 \omgam, \quad  a(u,\Ihlifte v) = l(\Ihlifte v) = \ahell (u_h^\ell, \Ihlifte v).
\end{equation}
\end{remark}
\begin{proof}[Proof of Lemma \ref{lem:Fh}]
Consider $v \in \H2(\Omega, \Gamma)$. We may decompose $|F_h(v)|$ in two terms as follows,
$$
|F_h(v)|  = |a(u-u_h^\ell, v)| 
          \le |a(u-u_h^\ell, v- \Ihlifte v)| + | a(u-u_h^\ell, \Ihlifte v)| 
          =: F_1 + F_2.
$$

Firstly, to bound $F_1$, we take advantage of the continuity of the bilinear form $a$ and apply the~$\HH1$ error estimation \eqref{errh1_errl2}, alongside the inequality~\eqref{ineq:interpolation-v}  as follows, 
\begin{align*}
    F_1  
         &\le c_{cont} \, \| u-u_h^\ell\|_{\HH1 \omgam} \| v-\Ihlifte v \|_{\HH1 \omgam}  
         \le c ( h^k + h^{r+1/2} ) \, h \|v\|_{\H2 \omgam} \\
         & \le c ( h^{k+1} + h^{r+3/2} ) \, \|v\|_{\H2 \omgam}.
\end{align*}

Secondly, to estimate $F_2$, we resort to equations \eqref{eq:a=ahell-For-Iv} and \eqref{ineq:a-al2} as follows,
\begin{multline*}
    F_2 = |a(u, \Ihlifte v)-a(u_h^\ell, \Ihlifte v)| 
         = |\ahell(u_h^\ell, \Ihlifte v)-a(u_h^\ell, \Ihlifte v)| 
         = |(\ahell-a)(u_h^\ell, \Ihlifte v)| \\
         \le c h^r \|\na u_h^\ell\|_{\L2(B_h^\ell)}\|\na (\Ihlifte v)\|_{\L2(B_h^\ell)} +c h^{r+1}\|u_h^\ell \|_{\HH1\omgam} \| \Ihlifte v\|_{\HH1\omgam}.
\end{multline*}
Next, we will treat the first term in the latter inequality separately. We have,
\begin{eqnarray*}
   F_3 & := & h^r \|\na u_h^\ell\|_{\L2(B_h^\ell)}\|\na (\Ihlifte v)\|_{\L2(B_h^\ell)} \\
       & \le & h^r \Big( \|\na (u_h^\ell-u) \|_{\L2(B_h^\ell)} + \|\na u\|_{\L2(B_h^\ell)} \Big) \Big( \|\na (\Ihlifte v-v)\|_{\L2(B_h^\ell)}+\|\na v\|_{\L2(B_h^\ell)} \Big)\\
       & \le & h^r \Big( \| u_h^\ell-u \|_{\HH1 \omgam} + \|\na u\|_{\L2(B_h^\ell)} \Big) \Big( \|\Ihlifte v-v\|_{\HH1 \omgam}+\|\na v\|_{\L2(B_h^\ell)} \Big).
\end{eqnarray*}
We now apply the $\HH1$ error estimation \eqref{errh1_errl2}, the inequality \eqref{ineq:blh-for-v} 
and the interpolation inequality~\eqref{ineq:interpolation-v}, as follows,
\begin{eqnarray*}
   F_3 & \le & c \, h^r \Big( h^k + h^{r+1/2} + h^{1/2} \| u\|_{\H2\omgam} \Big) \Big( h\|v\|_{\H2 \omgam}+h^{1/2}\| v\|_{\H2 \omgam} \Big)\\
   & \le & c \, h^r \, h^{1/2} \Big( h^{k-1/2} + h^{r} + \| u\|_{\H2\omgam} \Big) \Big( h^{1/2}+1\Big) h^{1/2} \| v\|_{\H2 \omgam} \\
    & \le & c \, h^{r+1} \Big( h^{k-1/2} + h^{r} + \| u\|_{\H2\omgam} \Big) \Big( h^{1/2}+1\Big) \| v\|_{\H2 \omgam}.
\end{eqnarray*}
Noticing that $k-1/2>0$ (since $k\ge 1$) and that $\Big( h^{k-1/2} + h^{r} + \| u\|_{\H2\omgam} \Big) \Big( h^{1/2}~+~1~\Big)$ is bounded by a constant independent of $h$, 
we obtain 
$ F_3 \le c \, h^{r+1} \| v\|_{\H2 \omgam}. $
Using the previous expression of $F_2$, 
$$
    F_2  \le  c  h^{r+1} \| v\|_{\H2 \omgam} +c h^{r+1}\|u_h^\ell \|_{\HH1\omgam} \| \Ihlifte v\|_{\HH1\omgam}.
$$
Moreover, noticing that $ \| \Ihlifte v\|_{\HH1\omgam} \le c \| v\|_{\H2\omgam}$, 
$$
    F_2   \le  c  h^{r+1} \| v\|_{\H2 \omgam} +c h^{r+1}\|u_h^\ell \|_{\HH1\omgam}\| v\|_{\H2 \omgam} \le  c h^{r+1} \|  v\|_{\H2\omgam},
$$
using \eqref{rem:bound-uhell-indep-h}. We conclude the proof by summing the estimates of $F_1$ and~$F_2$.
\end{proof}

\begin{proof}[Proof of the $\L2$ estimate \eqref{errh1_errl2}]
Defining $e:=u-u^\ell_h$, the aim is to estimate the following $\L2$ error norm:
$
    \|e\|_{\L2\omgam}^2 = \|u-u^\ell_h\|_{\L2(\Omega)}^2 + \|u-u^\ell_h\|_{\L2(\Gamma)}^2.
$
Let $v\in \L2\omgam$. We define the following problem: find $z_{v} \in~\HH1\omgam$ such that,
\begin{equation}
\label{dual}
    a(w,z_{v}) =\langle w,v \rangle_{\L2\omgam}, \ \ \ \forall \ w \in \HH1\omgam, 
\end{equation}
Applying Theorem~\ref{th_existance_unicite_u} for $f=v$ and $g= v_{|_{\Gamma}}$, 
there exists a unique solution $z_v \in \HH1 \omgam$ to \eqref{dual}, which satisfies the following inequality,
$$
    \|z_{v}\|_{\H2\omgam} \le c \|v\|_{\L2\omgam}.
$$
Taking $v = e \in \L2\omgam$ and~$w = e \in ~\HH1 \omgam$  in \eqref{dual}, we obtain
$ F_h(z_e)=a(e,z_e)= \|e\|_{\L2\omgam}^2$. 
In this case, Theorem \ref{th_existance_unicite_u} implies, 
\begin{equation}
\label{reg}
    \|z_e\|_{\H2\omgam} \le c \|e\|_{\L2\omgam}.
\end{equation}
Applying Inequality \eqref{ineq:F-h} for $z_e \in \H2\omgam$ and afterwards Inequality \eqref{reg}, we have,
$$
    \|e\|_{\L2\omgam}^2=|F_h(z_e)| \le c ( h^{k+1} + h^{r+1} ) \|z_e\|_{\H2\omgam} \le c ( h^{k+1} + h^{r+1} ) \|e\|_{\L2\omgam},
$$
which concludes the proof.
\end{proof}

\section{Numerical experiments}
\label{sec:numerical-ex}
In this section are presented numerical results 
aimed to illustrate the theoretical convergence results
in Theorem \ref{th-error-bound}. Supplementary numerical results will be provided in order to highlight the properties of the volume lift introduced in definition \ref{def:liftvolume} relatively to the lift transformation~$\Ghr:~\omhr \rightarrow \Omega$  given in (\ref{eq:def-fter}). 

\medskip

All the numerical experiments presented here have been done using the finite element library for curved meshes CUMIN \cite{cumin}. Curved meshes of $\Omega$ of geometrical order~$1\le r \le 3$ have been generated using the software Gmsh\footnote{Gmsh: a three-dimensional finite element mesh generator, \url{https://gmsh.info/}. Additionally, all integral computations rely on quadrature rules on the reference elements which are always chosen of sufficiently high order: the integration errors have negligible influence over the forthcoming numerical results. All numerical results presented in this section can be fully reproduced using dedicated source codes available on CUMIN Gitlab\footnote{CUMIN GitLab deposit, \url{https://plmlab.math.cnrs.fr/cpierre1/cumin}}. }

\subsection{The two dimensional case}
The Ventcel problem (\ref{1}) is considered with~$\alpha = \beta = \kappa = 1$ 
on the unit disk $\Omega$,
$$
\arraycolsep=2pt
\left\{
\begin{array}{rcll}
-\Delta u + u &=& f     & \text{ in } \Omega, \\
{- \Delta_{\Gamma} u} + \partial_n u  + u &=& g & \text{ on } \Gamma,
\end{array}
\right.
$$
with the source terms $f(x,y)= - y \mathrm{e}^{x}$ and 
$g(x,y)=  y \mathrm{e}^{x} ( 3 + 4x-y^2 )$
corresponding to the exact solution
$u = -f$.


\medskip

{The numerical solutions $u_h$ are computed 
for $\mathbb{P}^k$ finite elements, with $k=1,\dots, 4$, on series of successively refined meshes of order $r=1,\dots, 3$, as depicted on 
figure~\ref{f1}
for coarse meshes (affine and quadratic). Each mesh counts $10\times 2^{n-1}$ edges on the domain boundary, for~$n=1\dots 7$. On the most refined mesh using a $\P 4$ finite element method, we counted~$10 \times 2^6$ boundary edges and approximately $75\, 500$ triangles. The associated~$\P 4$ finite element space has approximately $605\, 600$ DOF (Degrees Of Freedom). We mention that the computation time is very fast in the present case: total computations roughly last one minute on a simple laptop, which are made really efficient with the direct solver MUMPS\footnote{MUMPS, MUltifrontal Massively Parallel Sparse direct Solver, \url{https://mumps-solver.org/index.php}.} for sparse linear systems.}

%
%

%
%
%

%
%
\begin{figure}[H] 
\centering
    \includegraphics[width=0.25\textwidth]{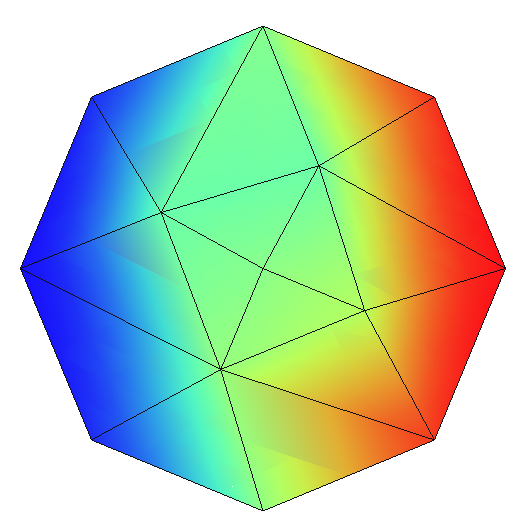}
    ~~~
    \includegraphics[width=0.25\textwidth]{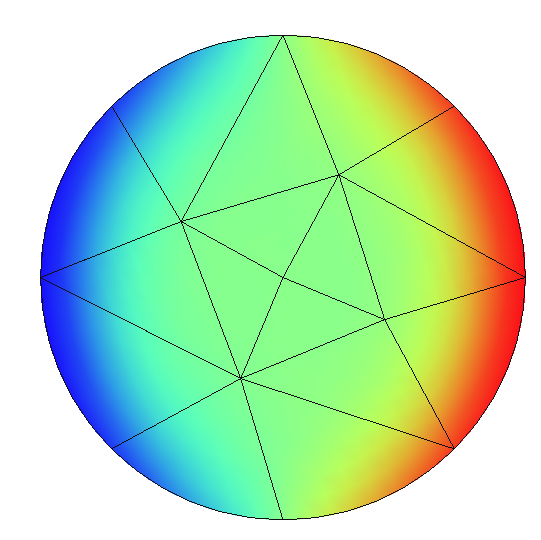}
 \caption{Numerical solution of the  Ventcel problem 
   on affine and quadratic meshes.}
\label{f1}
\end{figure}
In order to validate numerically the latter estimates, for each mesh order $r$ and each finite element degree $k$, the following numerical errors are computed on a series of refined meshes: 
\begin{displaymath}
   \| u-u_h^\ell \|_{\L2(\Omega)}, \quad 
   \| \nabla u- \nabla u_h^\ell \|_{\L2(\Omega) },\quad 
   \| u-u_h^\ell \|_{\L2(\Gamma)} \quad {\rm and}\quad 
   \| \nt u- \nt u_h^\ell \|_{\L2(\Gamma) }.
\end{displaymath}
The convergence orders of these errors,
interpreted in terms of the mesh size, are reported in Table~\ref{tab:conv-1-Omega}~and in Table~\ref{tab:conv-1-Gamma}. For readers convenience, these four errors are plotted with respect to the mesh size $h$ in Figure \ref{fig:error-2d-omega} with volume norms and in Figure~\ref{fig:error-2d-gamma} with surface norms. 

\begin{table}[!ht]
    \centering
    \begin{tabular}{|l||l|l|l|l||l|l|l|l|}
\cline{2-9}
\multicolumn{1}{c||}{}    &  \multicolumn{4}{|c||}{$\| u-u_h^\ell \|_{\L2(\Omega) }$} & \multicolumn{4}{|c|}{$\| \nabla  u-  \nabla  u_h^\ell \|_{\L2(\Omega)}$}  \\[0.08cm]
\cline{2-9}
\multicolumn{1}{c||}{}    &  $\P1$ &  $\P2$ &  $\P3$ &  $\P4$ & $\P1$ &  $\P2$ &  $\P3$ &  $\P4$  \tabularnewline
\hline
 Affine mesh (r=1)   & 
\textcolor{black}{1.98} & \textcolor{black}{1.99} & \textcolor{black}{1.97} & \textcolor{black}{1.97} & 
1.00 & \textcolor{black}{1.50} & \textcolor{black}{1.49} & \textcolor{black}{1.49} \tabularnewline
\hline
Quadratic mesh (r=2) & 
2.01 & 3.14 & \textcolor{black}{3.94} & \textcolor{black}{3.97} & 1.00 & 2.12 & 3.03 & \textcolor{black}{3.48} \tabularnewline
\hline
Cubic mesh (r=3) & 
2.04 & \textcolor{black}{2.45} & \textcolor{black}{3.44} & 4.04 & 
1.02 & \textcolor{black}{1.47} & \textcolor{black}{2.42} & 3.46 \tabularnewline
\hline
\end{tabular}

\caption{
  Convergence orders, interior norms. 
  \label{tab:conv-1-Omega}
}
\end{table}
\begin{figure}[h]
    \begin{center}
    \begin{tabular}{ c c } 
 \includegraphics[width=5cm, height=4cm]{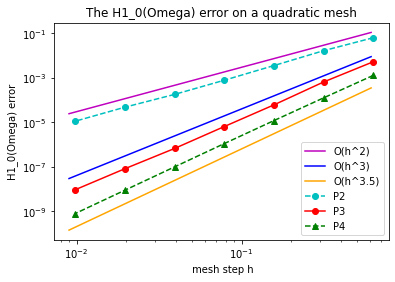} & \includegraphics[width=5cm, height=4cm]{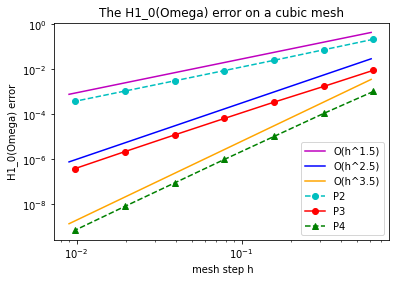}  \\ \includegraphics[width=5cm, height=4cm]{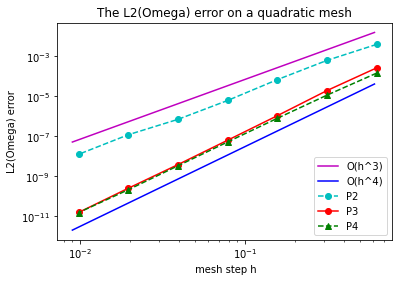} &
 \includegraphics[width=5cm, height=4cm]{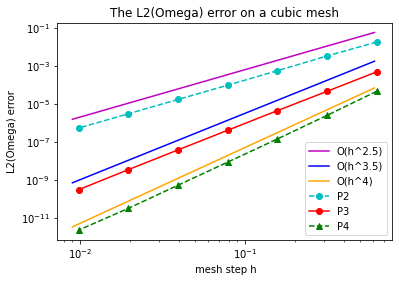}
\end{tabular}
\end{center}
     \caption{Plots of the error in volume norms with respect to the mesh step h corresponding to the convergence order in Table \ref{tab:conv-1-Omega}:
       ${\rm H}^1_0(\Omega)$ norm (above) and ${\rm L}^2(\Omega)$ norm (below) for quadratic meshes (left) and cubic meshes (right).}
     \label{fig:error-2d-omega}
 \end{figure}

The convergence orders presented in Table \ref{tab:conv-1-Omega} and in Figure \ref{fig:error-2d-omega}, relatively to $\L2$ norms on $\Omega$, deserve comments. In the affine case $r=1$, the figures are in perfect agreement with estimates~\eqref{errh1_errl2}: the $\L2$ error norm is in $O(h^{k+1} + h^2)$ and the $\L2$ norm of the gradient of the error is in  
$O(h^{k} + h^{1.5})$.

For quadratic meshes, a super convergence is observed in the geometric error, the case $r=2$ behaves as if  $r=3$: the $\L2$ error norm is in $O(h^{k+1} + h^4)$
and the~$\L2$ norm of the gradient of the error is in  
$O(h^{k} + h^{3.5})$. {This is quite visible in Figure \ref{fig:error-2d-omega} (left) for the $\L2$ error: while using respectively a $\P 3$ and $\P 4$ method, the $\L2$ error graphs in both cases follow the same line representing $O(h^4)$. In the case of the $\L2$ gradient norm of the error, this super convergence is depicted with a $\P3$ (resp. $\P4$) method: the convergence order is equal to 3 (resp. 3.5) surpassing the expected value of $2.5$.}
This super convergence, though not understood, has been documented and further investigated in \cite{D4,Jaca}.
{It has in particular to be noted that the super-convergence does not seem to be restricted neither to the present problem nor to the disk geometry considered here. Further numerical investigations showed that the geometric error relative to quadratic meshes and for integral computations is in $O(h^4)$ for various non-convex domains with no symmetry. In the next section, we will also see that it also holds in dimension 3.}

For the cubic case eventually, the $\L2$ error norm is expected to be in $O(h^{k+1/2} + h^4)$ and the~$\L2$ 
norm of the gradient of the error in 
$O(h^{k-1/2} + h^{3.5})$. This is accurately observed for a $\P1$ (resp.~$\mathbb{P}^4$) method: the $\L2$ error is equal to $2.04$ (resp. $4.04$) and the $\L2$ gradient error is equal to~$1.02$ (resp.~$3.46$). However, a default of order -1/2 is observed on the convergence orders in the $\mathbb{P}^2$ and $\mathbb{P}^3$ case. This default might not be in relation with the finite element approximation since it is not observed when considering $\L2(\Gamma)$ errors as shown in Table \ref{tab:conv-1-Gamma} and as discussed later on. Further experiments showed us that this default is not caused by the specific Ventcel boundary condition, it similarly occurs when considering a Poisson problem with Newman boundary condition on the disk. We also have experienced that this default of convergence is not related to the lift: actually it is related to the finite element interpolation error: so far we have no clues on its explanation.
\begin{table}[!ht]
    \centering
    \begin{tabular}{|l||l|l|l|l||l|l|l|l|}
\cline{2-9}
\multicolumn{1}{c||}{}    &  \multicolumn{4}{|c||}{$\| u-u_h^\ell \|_{\L2 (\Gamma)}$} & \multicolumn{4}{|c|}{$\| \nt  u-  \nt u_h^\ell \|_{\L2 (\Gamma)}$}  \\[0.08cm]
\cline{2-9}
\multicolumn{1}{c||}{}    &  $\P1$ &  $\P2$ &  $\P3$ &  $\P4$ & $\P1$ &  $\P2$ &  $\P3$ &  $\P4$  \tabularnewline
\hline
 Affine mesh (r=1)   & 
\textcolor{black}{2.00} & \textcolor{black}{2.03} & \textcolor{black}{2.01} & \textcolor{black}{2.01} & 
1.00 & \textcolor{black}{2.00} & \textcolor{black}{1.98} & \textcolor{black}{1.98} \tabularnewline
\hline
Quadratic mesh (r=2) & 
2.00 & 3.00 & \textcolor{black}{4.00} & \textcolor{black}{4.02} & 1.00 & 2.00 & 3.00 & \textcolor{black}{4.02} \tabularnewline
\hline
Cubic mesh (r=3) & 
2.00 & \textcolor{black}{3.00} & \textcolor{black}{4.00} & 4.21 & 
1.00 & \textcolor{black}{2.00} & \textcolor{black}{3.00} & 3.98 \tabularnewline
\hline
\end{tabular}
\caption{
  Convergence orders, boundary norms.
  \label{tab:conv-1-Gamma}
}
\end{table}
\begin{figure}[h]
    \begin{center}
    \begin{tabular}{ c c } 
\includegraphics[width=5cm, height=4cm]{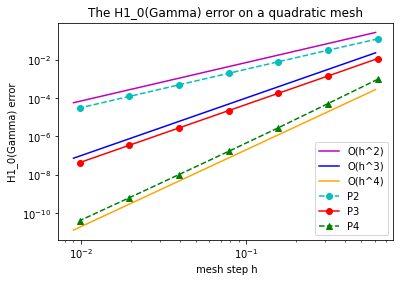} & \includegraphics[width=5cm, height=4cm]{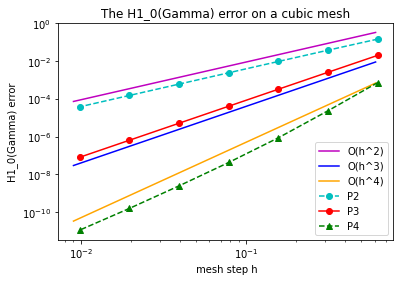} \\ 
 \includegraphics[width=5cm, height=4cm]{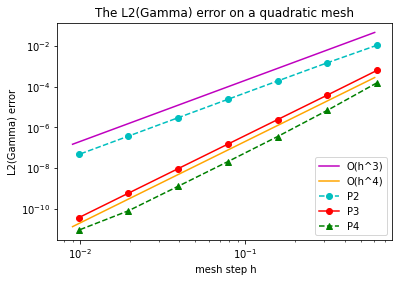} & \includegraphics[width=5cm, height=4cm]{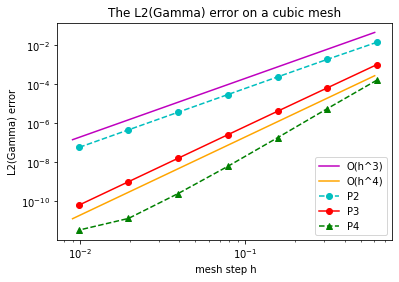} 
\end{tabular}
\end{center}
     \caption{Plots of the error in interior norms with respect to the mesh step h corresponding to the convergence order in Table \ref{tab:conv-1-Gamma}:
       ${\rm H}^1_0(\Gamma)$ norm (above) and ${\rm L}^2(\Gamma)$ norm (below) for quadratic meshes (left) and cubic meshes (right)..}
     \label{fig:error-2d-gamma}
 \end{figure}

\medskip

Let us now discuss Table \ref{tab:conv-1-Gamma} and Figure \ref{fig:error-2d-gamma}, where the surface errors and their convergence rates are observed. The first interesting point is that the $\L2$ convergence towards the gradient of $u$ is faster than expressed in (\ref{errh1_errl2}): $O(h^{k} + h^{r+1})$ instead of~$O(h^{k} + h^{r+1/2})$, as expected. Indeed, this is observed on a cubic and quadratic mesh with a $\P4$ method: the convergence rate is equal to $4$ instead of $3.5$. It seems that the estimate in Theorem \ref{th-error-bound} is not optimal for the tangential gradient norm on~$\Gamma$: so far we have not been able to improve it. Meanwhile the $\L2$ convergence towards~$u$ behaves as expected. Additionally, the super-convergence previously described for quadratic meshes is clearly visible for the boundary norms too. We also notice that the default of convergence of magnitude -1/2 for cubic meshes is absent here.
 \medskip

\paragraph{Lift transformation regularity.} 
{
In Remark \ref{rem:lift-regularity}, we discussed the dependency of the regularity of the lift transformation~$\Ghr:~\omhr \rightarrow \Omega$ defined in \eqref{eq:def-fter} with respect to the exponent $s$ in the term~$(\lambda^\star)^s$. According to the theory, the exponent $s$ in~$(\lambda^\star)^s$ needs to be set to $r+2$ to ensure that $\Ghr$ is piece-wise $C^{r+1}$ on each element. In theory, it is thus necessary to set $s=r+2$ for the estimates in Theorem 
\ref{th-error-bound} to hold. Surprisingly, we have remarked that in practice, 
estimates in Theorem 
\ref{th-error-bound} still hold when decreasing the exponent of $s$ of $(\lambda^\star)^s$.
When setting $s=2$, the results in Table~\ref{tab:conv-1-Omega} and in Table \ref{tab:conv-1-Gamma} remain unchanged.
When setting $s=1$, the same conclusion holds, though in this case $\diff {\Ghr} $ has singularities on the non-internal elements.
This is quite surprising since the estimate in \eqref{ineq:Gh-Id_Jh-1}, which is crucial for the error analysis, no longer holds.
Beyond the convergence rate, we have also
noticed that the accuracy itself is not
damaged when decreasing the exponent $s$ of $(\lambda^\star)^s$. 
A plausible reason for this is that the singular points of the derivatives of~$G_h^{(r)}$ are always located at one element vertex or edge. They are ``\textit{not seen}'', likely because they are away from the quadrature method nodes (used to approximate the integrals) that are located in the interior of considered element. Consequently, the singularities are not detected by this method.}

\begin{table}[!ht]
    \centering
    \begin{tabular}{|l||l|l|l|l||l|l|l|l|}
\cline{2-9}
\multicolumn{1}{c||}{}    &  \multicolumn{4}{|c||}{$\| u-\textcolor{blue}{u_h^{e\ell}} \|_{\L2(\Omega) }$} & \multicolumn{4}{|c|}{$\| \nabla  u-  \nabla  \textcolor{blue}{u_h^{e\ell}} \|_{\L2(\Omega)}$}  \\[0.08cm]
\cline{2-9}
\multicolumn{1}{c||}{}    &  $\P1$ &  $\P2$ &  $\P3$ &  $\P4$ & $\P1$ &  $\P2$ &  $\P3$ &  $\P4$  \tabularnewline
\hline
Quadratic mesh (r=2) & 
2.01 & 2.51 & \textcolor{black}{2.49} & \textcolor{black}{2.49} & 1.00 & 1.52 & 1.49& \textcolor{black}{1.49} \tabularnewline
\hline
Cubic mesh (r=3) & 
2.04 & \textcolor{black}{2.50} & \textcolor{black}{2.48} & 2.49 & 
1.03 & \textcolor{black}{1.51} & \textcolor{black}{1.49} & 1.49 \tabularnewline
\hline
\end{tabular}

\begin{tabular}{c}
  \\
\end{tabular}

\begin{tabular}{|l||l|l|l|l||l|l|l|l|}
\cline{2-9}
\multicolumn{1}{c||}{}    &  \multicolumn{4}{|c||}{$\| u-\textcolor{blue}{u_h^{e\ell}} \|_{\L2 (\Gamma)}$} & \multicolumn{4}{|c|}{$\| \nt u-  \nt \textcolor{blue}{u_h^{e\ell}} \|_{\L2 (\Gamma)}$}  \\[0.08cm]
\cline{2-9}
\multicolumn{1}{c||}{}    &  $\P1$ &  $\P2$ &  $\P3$ &  $\P4$ & $\P1$ &  $\P2$ &  $\P3$ &  $\P4$  \tabularnewline
\hline
Quadratic mesh (r=2) & 
2.00 & 3.00 & \textcolor{black}{2.99} & \textcolor{black}{2.99} & 1.00 & 2.00 & 3.00 & \textcolor{black}{2.98} \tabularnewline
\hline
Cubic mesh (r=3) & 
2.00 & \textcolor{black}{3.00} & \textcolor{black}{2.99} & 2.98 & 
1.00 & \textcolor{black}{2.00} & \textcolor{black}{3.00} & 2.98 \tabularnewline
\hline
\end{tabular}
\caption{
  Convergence orders for the lift in \cite{elliott}.
  \label{tab:conv-2}
}
\end{table}
%
 
\medskip

\paragraph{Former lift definition.} 
As developed in remark \ref{rem:lift-elliott-trace-non-conserve}, another lift transformation $G_h:~\omhr \rightarrow \Omega$  had formerly been introduced in \cite{elliott}, with different properties on the boundary. We reported the convergence orders observed with this lift in Table \ref{tab:conv-2}. \medskip

The first observation is that $\| u- \textcolor{blue}{u_h^{e\ell}} \|_{\L2 (\Omega)}$ is at most in $O(h^{2.5})$ whereas $\| \nabla u-  \nabla \textcolor{blue}{u_h^{e\ell}} \|_{\L2 (\Omega)}$ is at most in $O(h^{1.5})$, resulting in a clear decrease of the convergence rate as compared to tables \ref{tab:conv-1-Omega} and \ref{tab:conv-1-Gamma}. Similarly, $\| u-  \textcolor{blue}{u_h^{e\ell}} \|_{\L2 (\Gamma)}$ and $\| \nabla u-  \nabla \textcolor{blue}{u_h^{e\ell}} \|_{\L2 (\Gamma)}$ are at most in $O(h^3)$ whereas they could reach $O(h^4)$ in tables \ref{tab:conv-1-Omega} and \ref{tab:conv-1-Gamma}. \medskip

Notice that the lift transformation intervenes at two different stages: for the right hand side definition in (\ref{fvh}) and for the error computation itself. We experienced the following. We set the lift for the right hand side computation to the one in \cite{elliott} whereas the lift for the error computation is the one in definition \ref{def:liftvolume}  (so that the numerical solution $u_h$ is the same as in Table \ref{tab:conv-2}, only its post treatment in terms of errors is different). Then we observed that the results are partially improved: for the~$\mathbb{P}^4$ case on cubic meshes, $\| u-  \textcolor{blue}{u_h^{e\ell}} \|_{\L2 (\Omega)} = O(h^{3.0})$ and $\| \nabla u-  \nabla \textcolor{blue}{u_h^{e\ell}} \|_{\L2 (\Omega)} = O(h^{2.5})$, which remain lower than the convergence orders in Table \ref{tab:conv-1-Omega}.\medskip

Still considering the lift definition in \cite{elliott}, we also experienced that the exponent~$s$ in the term~$(\lambda^\star)^s$ in the lift definition (see remark \ref{rem:lift-regularity}) has an influence on the convergence rates. Surprisingly, the best convergence rates are obtained when setting $s=1$: this case corresponds to the minimal regularity on the lift transformation $G_h$, the differential of which (as previously discussed) has singularities on the non-internal mesh elements. In that case however, the convergence rares goes up to $O(h^{3.5})$ and $O(h^{2.5})$ on quadratic and cubic  meshes for $\| u-\textcolor{blue}{u_h^{e\ell}} \|_{\L2(\Omega) }$ and $\| \nabla  u-  \nabla  \textcolor{blue}{u_h^{e\ell}} \|_{\L2(\Omega)}$ respectively. Meanwhile, it has been noticed that setting $s=1$ somehow damages the quality of the numerical solution on the domain boundary: these last results are surprising and with no clear explanation. Eventually, when setting $s\ge 2$, the convergence rates are lower and identical to those in Table \ref{tab:conv-2}.
 
\subsection{A 3D case: error estimates on the unit ball}
The system \eqref{1} is considered on the unit ball $\Omega = {\rm B(O,1)} \subset \R^3$, with source terms $f=-(x+y)\mathrm{e}^{z}$ on the domain and $g=(x+y)\mathrm{e}^{z}(5z+z^2+3)$ on the boundary. The ball is discretized using meshes of order~$r=1,\dots, 3$, which are depicted in Figure~\ref{f3d} for affine and quadratic meshes. 
\begin{figure}[H]
\centering
    \includegraphics[width=0.3\textwidth]{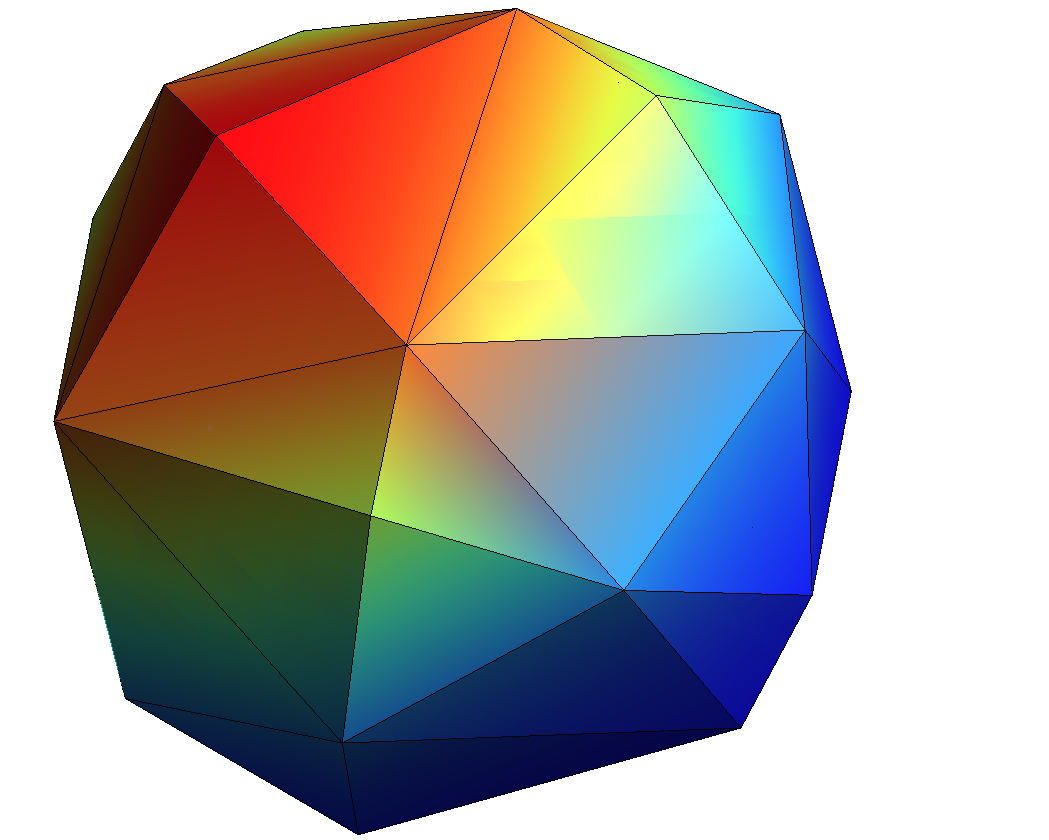}
    ~~~
    \includegraphics[width=0.3\textwidth]{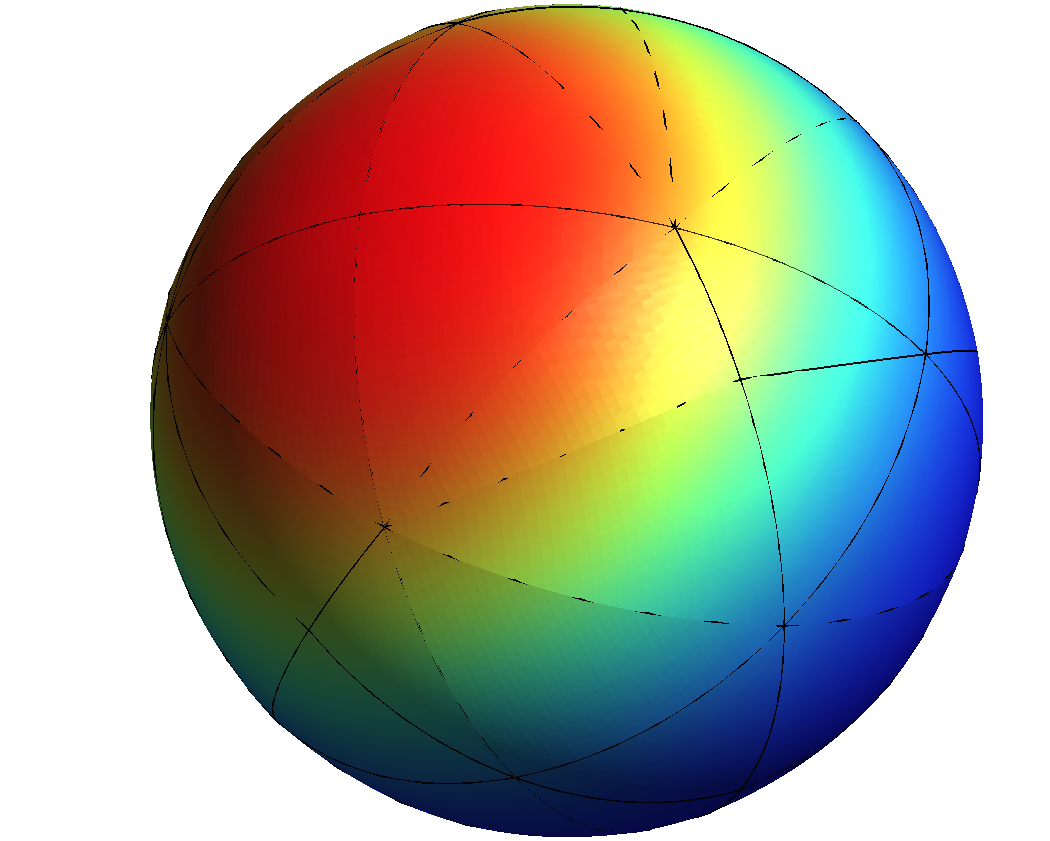}
 \caption{Numerical solution of the Ventcel problem on affine and quadratic meshes.}
 \label{f3d}
\end{figure}

For each mesh order $r$ and finite element degree $k$, we compute the error on a series of six successively refined meshes. Each mesh counts~$10\times 2^{n-1}$ edges on the equator circle, for~$n=1, \dots ,6$. The most refined mesh has approximately $2,4 \times 10^6$ tetrahedra and the associated $\P 3$ finite element method counts  $11 \times 10^6$ degrees of freedom. Consequently the matricial system of the spectral problem, which needs to be solved, has a size~$11 \times 10^6 $ with a rather large stencil. As a result, in the 3D case, the computations are much more demanding. The use of MUMPS, as we did in the 2D case, is no longer an option due to memory limitation.  The inversion of the linear system is done using the conjugate gradient method with a Jacobi pre-conditioner. To handle these computations, we resorted to the UPPA research computer cluster PYRENE\footnote{PYRENE Mesocentre de Calcul Intensif Aquitain, \url{https://git.univ-pau.fr/num-as/pyrene-cluster}.}. Using shared memory parallelism on a single CPU with~$32$~cores and~$2\, 000$~Mb of memory, the total time required is around $2$ hours. \medskip

The following numerical errors are computed on a series of refined meshes, using the lift defined in section \ref{sec:lift-def-surf-vol}: 
\begin{displaymath}
   \| u-u_h^\ell \|_{\L2(\Omega)}, \quad 
   \| \nabla u- \nabla u_h^\ell \|_{\L2(\Omega) },\quad 
   \| u-u_h^\ell \|_{\L2(\Gamma)} \quad {\rm and}\quad 
   \| \nt u- \nt u_h^\ell \|_{\L2(\Gamma) }.
\end{displaymath}

\begin{figure}[h]
\begin{center}
\begin{tabular}{ c c } 
\includegraphics[width=5cm, height=4cm]{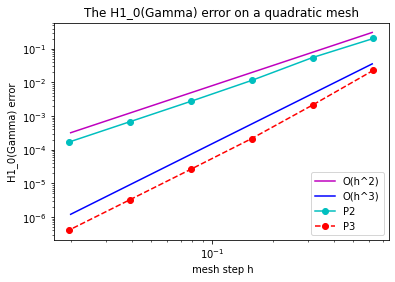} & \includegraphics[width=5cm, height=4cm]{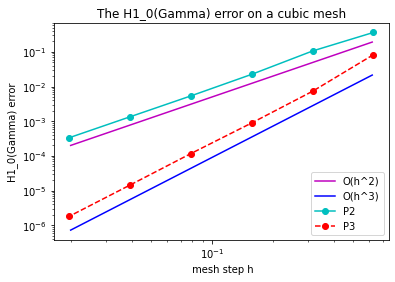} \\ 
\includegraphics[width=5cm, height=4cm]{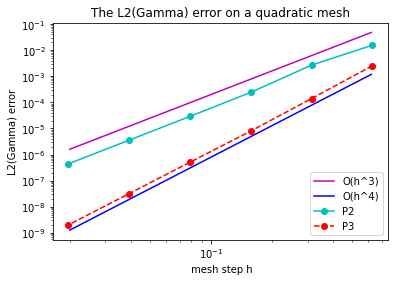} & \includegraphics[width=5cm, height=4cm]{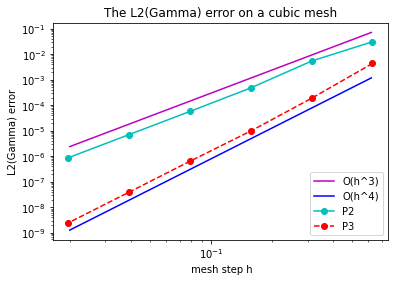} 
\end{tabular}
\end{center}
     \caption{3D case: plots of the error in $\HH1_0(\Gamma)$ norm (above) and $\L2(\Gamma)$ norm (below) 
and for quadratic meshes (left) and cubic meshes (right).
}
     \label{fig:error-3d-surface}
 \end{figure}

{In figure \ref{fig:error-3d-surface}, is displayed a log–log graph of each of the surface errors in $\HH1_0$ and~$\L2$ norms on quadratic and cubic meshes using $\P2$ and $\P3$ finite element methods. 
As a general comment: it can be seen that the quadratic meshes also exhibit a super-convergence as in dimension 2 and always behave as if $r=3$ instead of the expected~$r=2$.

As observed in the case of the disk, the $\L2$ surface errors behave quite well following the inequalities in  \eqref{errh1_errl2}. The $\HH1$ surface errors follow the same pattern as in the previous case: the error is in $O(h^{k} + h^{r+1})$ instead of $O(h^{k} + h^{r+1/2})$.}
 
 \begin{figure}[h]
\begin{center}
\begin{tabular}{ c c } 
\includegraphics[width=5cm, height=4cm]{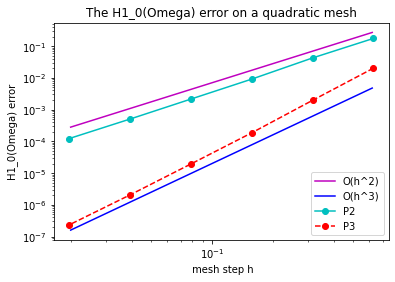} & \includegraphics[width=5cm, height=4cm]{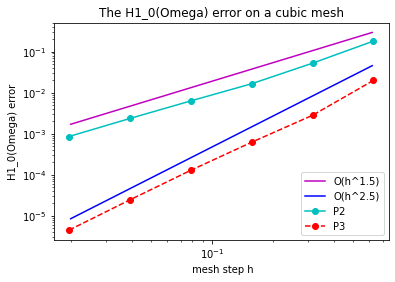} 
\end{tabular}
\end{center}
     \caption{3D case: plots of the error in $\HH1_0(\Omega)$ norm  for quadratic meshes (left) and cubic meshes (right).
}
     \label{fig:error-3d-H10}
 \end{figure}
{In Figure \ref{fig:error-3d-H10}, the $\HH1_0$ error in the volume is computed on quadratic meshes (left) and cubic meshes (right) with a $\P2$ and $\P3$ methods. In the quadratic case, the error has a convergence order of 2 (resp. 3) for a $\P2$ (resp. $\P 3$) method, following the inequality \eqref{errh1_errl2}. In the cubic case, the  same phenomena is observed as in the case of the disk: a loss of $-1/2$ in the convergence rate is detected, and the error is in $O(h^{1.5})$ (resp. $O(h^{2.5})$) for a $\P2$ (resp. $\P 3$) method. }
 
\begin{figure}[h]
\begin{center}
\begin{tabular}{ c c } \includegraphics[width=5cm, height=4cm]{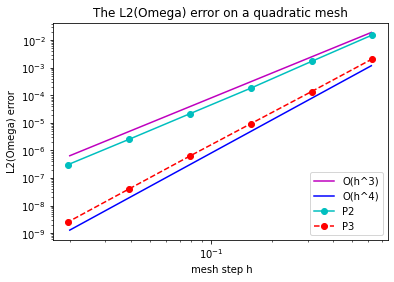} &
\includegraphics[width=5cm, height=4cm]{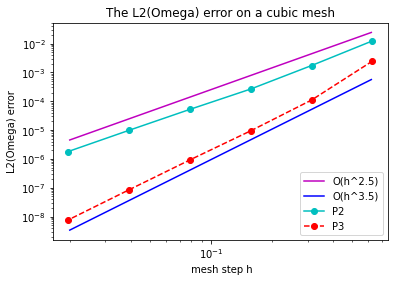} 
\end{tabular}
\end{center}
     \caption{
3D case: plots of the error in $\LL2(\Omega)$ norm  for quadratic meshes (left) and cubic meshes (right).
}
     \label{fig:error-3d-L2}
 \end{figure}
{In Figure \ref{fig:error-3d-L2}, the $\L2$ error in the volume is computed on quadratic meshes (left) and cubic meshes (right) with a $\P2$ and $\P3$ methods. In the quadratic case, the error has a convergence order of 3 (resp. 4) for a $\P2$ (resp. $\P 3$) method. This indicates that the super convergence phenomena is still observed on 3D domains. In the cubic case, the  same default of $-1/2$ in the convergence rate is observed as in the case of the disk: the graph of the error seems to have a slope of $2.5$ (resp. $3.5$) instead of $3$ (resp. $4$) for a $\P2$ (resp. $\P 3$) method. }

\appendix

%
%
%
%
%
\section{Proof of Proposition \ref{prop-Gh}}
\label{appendix:proof-Ghr}
%
%
%
%
%
%
%
%
%
%
%
%
%
Following the notations given in definition~\ref{def:sigma-lambdaetoile-haty}, we present the proof of Proposition \ref{prop-Gh} which requires a series of preliminary results given in  Propositions \ref{prop-y}, \ref{Prop:by-y} and  \ref{prop:rho-tr}. The proofs of these propositions are inspired by the proofs of \cite[Lemma 6.2]{Bernardi1989}, ~\cite[Lemma 4.3]{elliott} and~\cite[proposition 4.4]{elliott} respectively.

\begin{proposition}
\label{prop-y}
The map $y: \hatx \in \trefminissigma \mapsto  y:=\ftr{( \haty)} \in \ghr$ is a smooth function and for all~$m \ge 1$, there exists a constant $c>0$ independent of $h$ such that,
\begin{equation}
\label{ineq:y}
   \|\diff^m y \|_{\Linf(\trefminissigma)} \le \frac{ch}{(\lambdaetoile)^m}.
\end{equation}
\end{proposition}
\begin{remark}
\label{rem:Faa-di-bruno}
    The proof of this proposition and of the next one rely on the formula of Faà di Bruno (see \cite[equation 2.9]{Bernardi1989}). This formula states that for two functions $f$ and $g$, which are of class~$\c m$, such that $f \circ g$ is well defined, then, 
\begin{equation}
\label{Faa_di_bruno_eqt}
     \diff^m(f\circ g) = \sum_{p=1}^m \Big(  \diff^p(f) \sum_{i\in E(m,p)} c_i \prod_{q=1}^m   \diff^q  g ^{i_q} \Big), 
\end{equation}
where $E(m,p) := \{ i \in \N^{m};  \sum_{q=1}^m i_q=p$ and $\sum_{q=1}^m q i_q = m \}$ and $c_i$ are positives constants, for all~$i \in E(m,p)$.
\end{remark}
%
%
%
%
%
%
\begin{proof}[Proof of Proposition \ref{prop-y}]
We detail the proof in the $2$ dimensional case, the~3D case can be proved in a similar way. 

Consider, the reference triangle~$\tref$ with the usual orientation. Its vertices are denoted $(\hatv_i)_{i=1}^3$ and the associated barycentric coordinates respectively are: $\lambda_1 = 1-x_1-x_2$, $\lambda_2 = x_2$ and $\lambda_3 = x_1$. Consider a non-internal mesh element $\tr$ such that, without loss of generality, $v_1 \notin \Gamma$. In such a case, depicted in figure \ref{E-Er}, $\varepsilon_1=0$ and  $\varepsilon_2=\varepsilon_3=1$, since~$v_2, v_3~\in~\Gamma~\cap~\tr$. This implies that~$\lambdaetoile=\lambda_2+\lambda_3=x_2+x_1$ and,
\begin{equation}
 \label{eq:haty}   
    \haty= \frac{1}{\lambdaetoile} (\lambda_2 \hatv_2+\lambda_3 \hatv_3) = \frac{1}{x_2+x_1} (x_2 \hatv_2+x_1 \hatv_3).
\end{equation}
In this case, $\hatsigma= \{ \hatv_1\}$ and $\haty$ is defined on $\tref \setminus \{ \hatv_1\}$. 

%
%
%
%
%
%
\begin{figure}[h]
\centering
\begin{tikzpicture}

\draw (0.5,0.5) node[above] {$\hat{T}$};

\draw (0,0) node  {$\bullet$};
\draw (2,0) node  {$\bullet$};
\draw (0,2) node  {$\bullet$};

\draw (0,0) node[below] {$\hatv_1$};
\draw (2,0) node[below]  {$\hatv_2$};
\draw (0,2) node[above]  {$\hatv_3$};

\draw (0,0) -- (2,0) ;
\draw (0,0) -- (0,2);
\draw[blue] (2,0) -- (0,2) ;


\draw (1,0) node  {$\bullet$};
\draw (1,1) node  {$\bullet$};
\draw (0,1) node  {$\bullet$};





\draw[thick, red] [->] (2.8,1) -- (4.3,1);
\draw[red] (3.5,1) node[above] {$\textcolor{red}{\ftr}$};

\draw[blue] (1.3,0.7) node  {$\bullet$};
\draw[blue] (0.7,0.3) node  {$\bullet$};

\draw[blue] (1.3,0.7) node[below] {$ \haty$};
\draw[blue] (0.7,0.3) node[left] {$ \hatx$};

\draw[blue] (1.3,0.7) -- (0.7,0.3) ;





\draw[red] (7.4,0.8) node[above] {$\tr$};


\draw (6,0) node  {$\bullet$};
\draw (7.36,2.5) node  {$\bullet$};
\draw (9,0) node  {$\bullet$}; 
\draw (7.36,2.56) -- (9,0);
\draw (9,0) -- (6,0);


\draw (6,0)  node[left] {$v_2$};
\draw (7.36,2.6) node[above]  {$v_3$};
\draw (9,0) node[below]  {$v_1$};


\draw (7.5,0) node  {$\bullet$};
\draw (6.4,1.5) node  {$\bullet$}; 
\draw (8.15,1.3) node  {$\bullet$};


\draw[red] (6,0) arc (180:122:3);

\draw plot [domain=-0.2:1.5] (\x+6,-1.786*\x^2+4.3*\x);

\draw (5,-0.6) node {$\Gamma$};
\draw[thick]   [->] (5.2,-0.6)--(5.88,-0.5);





\draw[red] (6.07,0.7) node  {$\bullet$};
\draw[red] (6.7,0.3) node  {$\bullet$};

\draw[red] (6.07,0.7) node[left] {${y}$};
\draw[red] (6.7,0.3) node[right] {${x}$};

\draw[red] (6.7,0.3) -- (6.07,0.7) ;


\draw[thick,blue]   [->] (1.6,1.6)--(0.8,1.2);
\draw[blue] (1.8,1.7) node {$\hat{e}$};

\draw[red] (5.5,2.85) node {$e^{(r)}=\ftr(\hat{e})$};
\draw[thick,red]   [->] (5.9,2.6)--(6.8,2);

\end{tikzpicture}
\caption{Displaying $\ftr:\tref \to \tr$ in a 2D quadratic case (r=2).}
\label{E-Er}
\end{figure}
%
%
%
%
%

By differentiating the expression \eqref{eq:haty} of $\haty$ and using an induction argument, it can be proven that 
there exists a constant $c>0$, independent of $h$, such that, 
\begin{equation}
\label{ineq:haty}
    \|\diff^m  \haty \|_{\Linf(\trefminissigma)} \le \frac{c}{(\lambdaetoile)^m}, \ \ \ \ \ \mbox{ for all } m\ge 1. 
\end{equation}

Since $\ftr$ is the $\P r$-Lagrangian interpolant of $\fte$ on $\tref$, then 
$y=\ftr \circ \haty$ is a smooth function on $\trefminissigma$. 
%
%
%
%
%
%
%
%
%
%
We now apply the inequality \eqref{Faa_di_bruno_eqt} for $y = \ftr \circ \haty$ to estimate its derivative's norm as follows, for all $m\ge 1$,
\begin{equation*}
\label{bernardi_with_y}
    \| \diff^m(y)\|_{\Linf(\trefminissigma)} \le \sum_{p=1}^m \Big( \| \diff^p(\ftr)\|_{\Linf(\hat{e})} \sum_{i\in E(m,p)} c_i \prod_{q=1}^m  \| \diff^q  \haty\|_{\Linf(\trefminissigma)}^{i_q} \Big),
\end{equation*}
where $\hat{e}:= (\ftr)^{(-1)}(e\r)$ and $e\r:=\partial \tr \cap \ghr$ are displayed in Figure \ref{E-Er}. 
%
%
%
%
%
%
Afterwards, we decompose the sum into two parts, one part taking $p=1$ and the second one for $p \ge 2$, and apply inequality~\eqref{ineq:haty}, 
\begin{multline*}
\| \diff^m(y)\|_{\Linf(\trefminissigma)} \\
\begin{array}{l}
 \le  \displaystyle \| \diff(\ftr)\|_{\Linf(\hat{e})} \!\!\!\!\!\sum_{i\in E(m,1)} \prod_{q=1}^m  (\frac{c}{(\lambdaetoile)^q})^{i_q} \!+\!  \sum_{p=2}^m \Big( \| \diff^p(\ftr)\|_{\Linf(\hat{e})} \!\!\!\!\!\sum_{i\in E(m,p)} \prod_{q=1}^m  (\frac{c}{(\lambdaetoile)^q})^{i_q}\Big) \!\!\!\!\!\!\!\\
     \le  \displaystyle c h  {\lambdaetoile}^{(-\sum_{q=1}^m qi_q)}+ c \sum_{p=2}^m h^r {\lambdaetoile}^{(-\sum_{q=1}^m qi_q)}
     \le  c h  (\lambdaetoile)^{-m},
\end{array}
\end{multline*}
using that $\| \diff(\ftr)\|_{\Linf(\hat{e})} \le ch$ and $\| \diff^p(\ftr)\|_{\Linf(\hat{e})} \le ch^r$, for $2 \le p \le r+1$ (see \cite[page 239]{ciaravtransf}), 
where the constant $c>0$ is independent of $h$. This concludes the proof. 
%
%
%
%
%
%
\end{proof}

\begin{proposition}
\label{Prop:by-y}
Assume that $\Gamma$ is $\c {r+2}$ regular. Then the mapping $b\circ y: \hatx \in \trefminissigma \mapsto b(y(\hatx)) \in \Gamma$ is of class $\c{r+1}$. Additionally, for any $1\le m\le r+1$, there exists a constant $c>0$ independent of~$h$ such that, 
\begin{equation}
\label{ineq:by-y}
    \|\diff^m(b(y)- y) \|_{\Linf(\trefminissigma)} \le \frac{ch^{r+1}}{(\lambdaetoile)^m}.
\end{equation}
\end{proposition}
%
%
%
%
%
%
\begin{proof}
Since $\Gamma$ is $\c {r+2}$ regular, the orthogonal projection $b$ is a $\c{r+1}$ function on a tubular neighborhood of $\Gamma$ (see \cite[Lemma~4.1]{actanum} or \cite{demlow2019}). Consequently, following Proposition \ref{prop-y},~$b(y)-y$ is of class $\c {r+1}$ on $\trefminissigma$. 

Secondly, consider $1\le m \le r+1$. Applying the Faà di Bruno formula \eqref{Faa_di_bruno_eqt} for the function~$b(y)-y=(b-id) \circ y$, we have,
\begin{equation}
\label{eq:Faa_di_bruno_by-y}
    \| \diff^m(b(y)-y)\|_{\Linf(\trefminissigma)} \le \sum_{p=1}^m \Big( \| \diff^p(b-id)\|_{\Linf(e^{(r)})} \sum_{i\in E(m,p)} c_i \prod_{q=1}^m  \| \diff^q y\|_{\Linf(\trefminissigma)}^{i_q} \Big),
\end{equation}
where $e^{(r)}= \partial \tr \cap \ghr$ is displayed in Figure \ref{E-Er}. 
%
%
Notice that 
$b(v)= v$ for any $\P r$-Lagrangian interpolation nodes $v\in \Gamma \cap e^{(r)}$. 
Then $id_{|_{e\r}}$ is the $\P r$-Lagrangian interpolant of $b_{|_{e\r}}$. Consequently, the interpolation inequality can be applied as follows (see \cite{EG,Bernardi1989}),
$$
  \forall z \in e\r,  \quad \|\diff^p (b(z)-z) \|
    \le ch^{r+1-p} , \ \ \ \ \ \mbox{ for any } 0 \le p \le r+1.
$$
This interpolation result combined with \eqref{ineq:y} is replaced in \eqref{eq:Faa_di_bruno_by-y} to obtain,
\begin{multline*}
    \| \diff_{ \hatx}^m(b(y)-y)\|_{\Linf(\trefminissigma)} \le c \sum_{p=1}^m \Big( h^{r+1-p} \sum_{i\in E(m,p)} \prod_{q=1}^m  (\frac{h}{(\lambdaetoile)^q})^{i_q} \Big)\\
     \le c \sum_{p=1}^m \Big( h^{r+1-p} \frac{h^{\sum_{q=1}^m i_q}}{(\lambdaetoile)^{\sum_{q=1}^m qi_q}} \Big)
     \le c \sum_{p=1}^m \Big( h^{r+1-p} \frac{h^p}{(\lambdaetoile)^m} \Big)
     \le  c \frac{h^{r+1}}{(\lambdaetoile)^m},
\end{multline*}
where the constant $c>0$ is independent of $h$. 
This concludes the proof.
\end{proof}
Now, we introduce the mapping $\rhotr$, such that $\ftre=\ftr+\rhotr$ transforms $\tref$ into the exact triangle $\te$.
%
%
%
%
%
\begin{proposition}
\label{prop:rho-tr}
%
%
Let  $\rhotr : \hatx \in \tref \mapsto \rhotr(\hatx) \in \R^d$, be given by, $$\rhotr(\hatx):= \left\{  
\begin{array}{ll}
 0 & {\rm if } \, \hatx \in \hatsigma, \\
(\lambdaetoile)^{r+2} (b(y) - y) &  {\rm if } \, \hatx \in \trefminissigma.
\end{array}\right.$$  The mapping $\rhotr$ is of class $\c {r+1}$ on $\tref$ and there exist a constant $c>0$ independent of $h$ 
such that,
\begin{equation}
\label{ineq:rho_tr}
   \| \diff^m \rhotr \|_{\Linf(\tref)} \le ch^{r+1}, \ \ \ \mbox{ for } \ \  0 \le m \le r+1.
\end{equation}
\end{proposition}
\begin{proof}
The mapping $\rhotr$ is of class $\c{r+1}(\trefminissigma)$, being the product of equally regular functions. Consider $ 0 \le m \le r+1$.
Applying the Leibniz formula, we have,
\begin{align*}
  \diff^m{\rhotr}_{|_{\trefminissigma}} & = \diff^m ((\lambdaetoile)^{r+2} (b(y) - y) ) \\
  & = \sum_{i=0}^m \big(_{\ i}^m \big) (r+2)....(r+3-i) (\lambdaetoile)^{r+2-i}\diff^{m-i}(b(y)-y).
\end{align*}
Then applying \eqref{ineq:by-y}, we get, for $\hatx \in \trefminissigma,$
$$
   \| \diff^m \rhotr(\hatx) \| \le c  \sum_{i=0}^m  (\lambdaetoile)^{r+2-i} \frac{ch^{r+1}}{(\lambdaetoile)^{m-i}} \le c h^{r+1}  (\lambdaetoile)^{r+2-m} .
$$
Since $r+2-m >0$, $(\lambdaetoile)^{r+2-m}~\underset{\hatx \to \hatsigma}\longrightarrow~0$. Consequently, $\diff^m \rhotr$ can be continuously extended by~$0$ on $\hatsigma$ when $0\le m \le r+1$. Thus $\rhotr \in \c {r+1}$ and the latter inequality ensures \eqref{ineq:rho_tr}.
\end{proof}

We can now prove Proposition~\ref{prop-Gh},
as mentioned before, its proof relies on the previous propositions.

\begin{proof}[Proof of Proposition \ref{prop-Gh}]
Let $\tr \in \taur$ be a non-internal curved element. Let $x=\ftr(\hatx) \in \tr$ where $\hatx \in \tref$. Following the equation~\eqref{eq:def-fter}, we recall that, $\ftre( \hatx) = x +~\rhotr~(\hatx)$. Then $\Ghr$ can be written as follows,
$$
    {\Ghr}_{|_{\tr}}  = \ftre \circ ({\ftr})^{-1} 
     = (\ftr + \rhotr ) \circ ({\ftr})^{-1}  
     = id_{|_{\tr}} + \rhotr \circ ({\ftr})^{-1}.
$$

Firstly, with Proposition \ref{prop:rho-tr}, $\rhotr$ is of class $\c {r+1}(\tref)$ and $\ftr$ is a polynomial, then $\Ghr$ is also $\c {r+1}(\tr).$ 

Secondly, $\ftr$ is a $\c1$-diffeomorphism and there exists a constant $c>0$ independent of $h$ such that (see \cite[page 239]{ciaravtransf}),
\begin{equation}
    \label{ineq:ftr-1}
    \| \diff (\ftr)^{-1}\| \le \frac{c}{h}.
\end{equation}
Additionally, by applying \eqref{ineq:rho_tr} and \eqref{ineq:ftr-1}, the following inequality holds, 
\begin{equation}
    \label{ineq:cdt-ciar}
    \|\diff ( \rhotr)\|_{\Linf(\tref)} \| \diff(({\ftr})^{-1})\|_{\Linf(\tr)} \le ch^{r+1}   \frac{c}{h} = c h^r < 1.    
\end{equation}
Then by applying \cite[Theorem 3]{ciaravtransf}, $\ftr+\rhotr$ is a $\c1$-diffeomorphism, being the sum of a $\c 1$-diffeomorphism and a $\c 1$ mapping, which satisfy \eqref{ineq:cdt-ciar}. Therefore, $\Ghr=(\ftr+\rhotr) \circ (\ftr)^{-1}$ is a~$\c1$-diffeomorphism. 

To obtain the first inequality of \eqref{ineq:Gh-Id_Jh-1}, we differentiate the latter expression, 
$$
     \diff {\Ghr}_{|_{\tr}} - \I_{|_{\tr}} = \diff (\rhotr \circ ({\ftr})^{-1}) = \diff (\rhotr) \circ {(({\ftr})^{-1})} \diff({\ftr})^{-1}.
$$
Using \eqref{ineq:rho_tr} and \eqref{ineq:ftr-1}, we obtain,
$$
    \| \diff {\Ghr}_{|_{\tr}} - \I_{|_{\tr}}\|_{\Linf(\tr)} \le \|\diff ( \rhotr)\|_{\Linf(\tref)} \| \diff(({\ftr})^{-1})\|_{\Linf(\tr)} 
    \le ch^r,
$$
where the constant $c>0$ is independent of $h$. 
Lastly, the second inequality of~\eqref{ineq:Gh-Id_Jh-1} comes as a consequence of the first one, by definition of a Jacobian. 
\end{proof}
%
%
%
%
%
%
%
%
%

\bibliographystyle{abbrv}
\bibliography{biblio}

\begin{thebibliography}{10}

\bibitem{Bernardi1989}
C.~Bernardi.
\newblock Optimal finite-element interpolation on curved domains.
\newblock {\em SIAM J. Numer. Anal.}, 26(5):1212--1240, 1989.

\bibitem{D3}
A.~Bonito and A.~Demlow.
\newblock Convergence and optimality of higher-order adaptive finite element
  methods for eigenvalue clusters.
\newblock {\em SIAM J. Numer. Anal.}, 54(4):2379--2388, 2016.

\bibitem{demlow2019}
A.~Bonito, A.~Demlow, and R.~H. Nochetto.
\newblock Finite element methods for the {L}aplace-{B}eltrami operator.
\newblock In {\em Geometric partial differential equations. {P}art {I}},
  volume~21 of {\em Handb. Numer. Anal.}, pages 1--103. 2019.

\bibitem{D4}
A.~Bonito, A.~Demlow, and J.~Owen.
\newblock A priori error estimates for finite element approximations to
  eigenvalues and eigenfunctions of the {L}aplace-{B}eltrami operator.
\newblock {\em SIAM J. Numer. Anal.}, 56(5):2963--2988, 2018.

\bibitem{Gvial}
V.~Bonnaillie-No\"{e}l, D.~Brancherie, M.~Dambrine, F.~H\'{e}rau, S.~Tordeux,
  and G.~Vial.
\newblock Multiscale expansion and numerical approximation for surface defects.
\newblock In {\em C{ANUM} 2010, {$40^{\rm e}$} {C}ongr\`es {N}ational
  d'{A}nalyse {N}um\'{e}rique}, volume~33 of {\em ESAIM Proc.}, pages 22--35.
  EDP Sci., Les Ulis, 2011.

\bibitem{Brenner-scott}
S.~C. Brenner and L.~R. Scott.
\newblock {\em The mathematical theory of finite element methods}, volume~15 of
  {\em Texts in Applied Mathematics}.
\newblock Springer-Verlag, New York, 1994.

\bibitem{quasi-unif}
S.~C. Brenner and L.~R. Scott.
\newblock The mathematical theory of finite element methods.
\newblock 15:16,361, 2002.

\bibitem{Jaca}
F.~Caubet, J.~Ghantous, and C.~Pierre.
\newblock Numerical study of a diffusion equation with ventcel boundary
  condition using curved meshes.
\newblock {\em Monografías Matemáticas García de Galdeano}, 2023.

\bibitem{PHcia}
P.~G. Ciarlet.
\newblock {\em The finite element method for elliptic problems}, volume~40 of
  {\em Classics in Applied Mathematics}.
\newblock Society for Industrial and Applied Mathematics (SIAM), Philadelphia,
  PA, 2002.

\bibitem{ciaravtransf}
P.~G. Ciarlet and P.-A. Raviart.
\newblock Interpolation theory over curved elements, with applications to
  finite element methods.
\newblock {\em Comp. Meth. Appl. Mech. Eng.}, 1:217--249, 1972.

\bibitem{tubneig}
C.~Dapogny and P.~Frey.
\newblock Computation of the signed distance function to a discrete contour on
  adapted triangulation.
\newblock {\em Calcolo}, 49(3):193--219, 2012.

\bibitem{D1}
A.~Demlow.
\newblock Higher-order finite element methods and pointwise error estimates for
  elliptic problems on surfaces.
\newblock {\em SIAM J. Numer. Anal.}, 47(2):805--827, 2009.

\bibitem{D2}
A.~Demlow and G.~Dziuk.
\newblock An adaptive finite element method for the {L}aplace-{B}eltrami
  operator on implicitly defined surfaces.
\newblock {\em SIAM J. Numer. Anal.}, 45(1):421--442, 2007.

\bibitem{dubois}
F.~Dubois.
\newblock Discrete vector potential representation of a divergence-free vector
  field in three-dimensional domains: numerical analysis of a model problem.
\newblock {\em SIAM J. Numer. Anal.}, 27(5):1103--1141, 1990.

\bibitem{Dz88}
G.~Dziuk.
\newblock Finite elements for the {B}eltrami operator on arbitrary surfaces.
\newblock In {\em Partial differential equations and calculus of variations},
  volume 1357 of {\em Lecture Notes in Math.}, pages 142--155. Springer,
  Berlin, 1988.

\bibitem{actanum}
G.~Dziuk and C.~M. Elliott.
\newblock Finite element methods for surface {PDE}s.
\newblock {\em Acta Numer.}, 22:289--396, 2013.

\bibitem{ed}
D.~Edelmann.
\newblock Isoparametric finite element analysis of a generalized {R}obin
  boundary value problem on curved domains.
\newblock {\em SMAI J. Comput. Math.}, 7:57--73, 2021.

\bibitem{elliott}
C.~M. Elliott and T.~Ranner.
\newblock Finite element analysis for a coupled bulk-surface partial
  differential equation.
\newblock {\em IMA J. Numer. Anal.}, 33(2):377--402, 2013.

\bibitem{EG}
A.~Ern and J.-L. Guermond.
\newblock {\em Theory and practice of finite elements}, volume 159 of {\em
  Applied Mathematical Sciences}.
\newblock Springer-Verlag, New York, 2004.

\bibitem{GT98}
D.~Gilbarg and N.~S. Trudinger.
\newblock {\em Elliptic partial differential equations of second order}.
\newblock Classics in Mathematics. Springer-Verlag, Berlin, 2001.
\newblock Reprint of the 1998 edition.

\bibitem{Grisvard2011}
P.~Grisvard.
\newblock {\em Elliptic problems in nonsmooth domains}, volume~69 of {\em
  Classics in Applied Mathematics}.
\newblock Society for Industrial and Applied Mathematics (SIAM), Philadelphia,
  PA, 2011.

\bibitem{livreopt}
A.~Henrot and M.~Pierre.
\newblock {\em Variation et optimisation de formes: une analyse
  g{\'e}om{\'e}trique}, volume~48.
\newblock Springer Science \& Business Media, 2006.

\bibitem{ventcel1}
T.~Kashiwabara, C.~M. Colciago, L.~Ded\`e, and A.~Quarteroni.
\newblock Well-posedness, regularity, and convergence analysis of the finite
  element approximation of a generalized {R}obin boundary value problem.
\newblock {\em SIAM J. Numer. Anal.}, 53(1):105--126, 2015.

\bibitem{Balazs}
B.~Kov\'{a}cs and C.~Lubich.
\newblock Numerical analysis of parabolic problems with dynamic boundary
  conditions.
\newblock {\em IMA J. Numer. Anal.}, 37(1):1--39, 2017.

\bibitem{Lenoir1986}
M.~Lenoir.
\newblock Optimal isoparametric finite elements and error estimates for domains
  involving curved boundaries.
\newblock {\em SIAM J. Numer. Anal.}, 23(3):562--580, 1986.

\bibitem{aeroacoustics}
E.~Luneville and J.-F. Mercier.
\newblock Mathematical modeling of time-harmonic aeroacoustics with a
  generalized impedance boundary condition.
\newblock {\em ESAIM Math. Model. Numer. Anal.}, 48(5):1529--1555, 2014.

\bibitem{nedelec}
J.-C. N\'{e}d\'{e}lec.
\newblock Curved finite element methods for the solution of singular integral
  equations on surfaces in {$R^{3}$}.
\newblock {\em Comput. Methods Appl. Mech. Engrg.}, 8(1):61--80, 1976.

\bibitem{cumin}
C.~Pierre.
\newblock The finite element library {C}umin, curved meshes in numerical
  simulations.
\newblock {\em repository: https://plmlab.math.cnrs.fr/cpierre1/cumin},
  hal-0393713(v1), 2023.

\bibitem{scott-2}
R.~Scott.
\newblock Interpolated boundary conditions in the finite element method.
\newblock {\em SIAM J. Numer. Anal.}, 12:404--427, 1975.

\bibitem{Ventcel-1956}
A.~D. Ventcel.
\newblock Semigroups of operators that correspond to a generalized differential
  operator of second order.
\newblock {\em Dokl. Akad. Nauk SSSR (N.S.)}, 111:269--272, 1956.

\bibitem{Ventcel-1959}
A.~D. Ventcel.
\newblock On boundary conditions for multi-dimensional diffusion processes.
\newblock {\em Theor. Probability Appl.}, 4:164--177, 1959.

\end{thebibliography}

\end{document}